\documentclass[a4paper,twoside,10pt]{article}

\usepackage[english]{babel}
\usepackage[latin1]{inputenc}
\usepackage{lmodern,color}
\usepackage[T1]{fontenc}
\usepackage{indentfirst}
\usepackage{amsmath,amssymb}
\usepackage[percent]{overpic}
\usepackage[a4paper,top=3cm,bottom=3cm,left=3cm,right=3cm,%
bindingoffset=5mm]{geometry}
\usepackage{lmodern}
\usepackage[T1]{fontenc}
\usepackage{lettrine}
\usepackage{booktabs}

\usepackage{authblk}
\usepackage{emp}
\usepackage{amsthm}
\usepackage{amsmath,amssymb,amsthm}
\usepackage{braket}
\usepackage{chngpage}
\usepackage{cases}
\usepackage{graphicx}
\usepackage{epstopdf}
\usepackage{tabularx}
\usepackage{multirow}

\newcommand{\numberset}{\mathbb}
\newcommand{\N}{\numberset{N}}
\newcommand{\R}{\numberset{R}}

\newcommand{\B}{\numberset{B}}
\newcommand{\Pk}{\numberset{P}}
\renewcommand{\epsilon}{\varepsilon}
\renewcommand{\theta}{\vartheta}
\renewcommand{\rho}{\varrho}
\renewcommand{\phi}{\varphi}
\newcommand{\uu}{\mathbf{u}}
\newcommand{\vv}{\mathbf{v}}
\newcommand{\ww}{\mathbf{w}}
\newcommand{\ff}{\mathbf{f}}

\newcommand{\dd}{{\rm div}}
\newcommand{\gr}{\nabla}
\newcommand{\Gr}{\boldsymbol{\nabla}}

\newcommand{\cc}{{{\rm curl}}}
\newcommand{\CC}{\boldsymbol{{\rm curl}}}
\newcommand{\dl}{\boldsymbol{\Delta}}
\newcommand{\VV}{\mathbf{V}}

\newcommand{\ZZ}{\mathbf{Z}}
\newcommand{\cconv}{c_{\rm conv}}
\newcommand{\cskew}{c_{\rm skew}}
\newcommand{\crot}{c_{\rm rot}}
\newcommand{\cconvh}{c_{{\rm conv}, h}}
\newcommand{\cskewh}{c_{{\rm skew}, h}}
\newcommand{\croth}{c_{{\rm rot}, h}}

\def\P0{{\Pi^{0, E}_k}}

\def\PN{{\Pi^{\nabla, E}_k}}
\def\PNdue{{\Pi^{\nabla^2, E}_{k+1}}}
\def\PP0{{\boldsymbol{\Pi}^{0, E}_{k-1}}}

{\left\lbrace\begin{array}{@{}l@{}}}%
{\end{array}\right.}
\theoremstyle{definition}

\theoremstyle{remark}
\newtheorem{remark}{Remark}[section]

\theoremstyle{remark}
\newtheorem{test}{Test}[section]

\theoremstyle{plain}
\newtheorem{theorem}{Theorem}[section]
\newtheorem{proposition}{Proposition}[section]
\newtheorem{corollario}{Corollary}[section]
\newtheorem{lemma}{Lemma}[section]

\usepackage{lipsum}

\usepackage{authblk}
\usepackage{color}

\usepackage{setspace}

\author[1]{L. Beir\~ao da Veiga \thanks{lourenco.beirao@unimib.it}}

\author[2,3]{D. Mora \thanks{dmora@ubiobio.cl}}

\author[1]{G. Vacca \thanks{giuseppe.vacca@unimib.it}}

\affil[1]{Dipartimento di Matematica e Applicazioni,
Universit\`a degli Studi di Milano Bicocca,
Via Roberto Cozzi 55 - 20125 Milano, Italy}  

\affil[2]{Departamento de Matem{\'a}tica, 
Universidad del B{\'i}o-B{\'i}o,
Casilla 5-C, Concepci{\'o}n, Chile} 

\affil[3]{CI$^2$MA, Universidad de Concepci\'on, Concepci\'on, Chile
}

\title{ \textbf{The Stokes complex for Virtual
Elements with application to Navier--Stokes flows}}

\date{\today}

\begin{document}

\maketitle
\begin{abstract}
In the present paper, we investigate the underlying
Stokes complex structure of the Virtual Element Method
for Stokes and Navier--Stokes introduced in previous
papers by the same authors, restricting our attention to the two dimensional case.
We introduce a Virtual Element space $\Phi_h \subset H^2(\Omega)$ and
prove that the triad $\{\Phi_h, \VV_h, Q_h\}$
(with $\VV_h$ and $Q_h$ denoting the discrete
velocity and pressure spaces) is an exact Stokes complex.
Furthermore, we show the computability of the associated
differential operators in terms of the adopted degrees
of freedom and explore also a different discretization
of the convective trilinear form. The theoretical
findings are supported by numerical tests.
\end{abstract}

\section{Introduction}
\label{sec:1}

The Virtual Element Method (VEM) was introduced in \cite{volley,VEM-hitchhikers} as a generalization of the Finite Element Method that allows to use general polygonal and polyhedral meshes. The VEM enjoyed a good success in this recent years both in the mathematics and in the engineering communities, a very brief list of papers being 
\cite{
GTP14,
hourglass,
Berrone-VEM,
Benedetto-VEM-2,
2016stability,
Mora-Rivera-Rodolfo:2015,
Mora-Rivera-Velazquez:2018,
wriggers,
wriggers-2,
Bertoluzza,
brenner:2017,
fumagalli,
ill_cond_VEM3D,
vaccahyper,
Artioli-Sacco},
while for the specific framework of incompressible flows we refer to \cite{Antonietti-BeiraodaVeiga-Mora-Verani:20XX, Stokes:divfree, Bdv-Lovadina-Vacca:2018,
Gatica-1, Gatica-2, Stokes:nonconforme, Cinese-nonconforme}.
Virtual Elements are part of a wider community of
polytopal methods, that include other technologies
(some representative papers being 
\cite{
cangiani-georgoulis-houston:2014,
Antonietti-et-al:2017,
floater-gillette-sukumar:2014,
talischi.paulino-pereira-menezes:2010,
dipietro-krell:2018,
botti-dipetro-droniou:2018}).

It was soon recognized that the more general construction
of VEM, that is not limited to polynomial functions on the
elements, may allow for additional interesting features
in additional to polytopal meshing. An example can be
found in \cite{Stokes:divfree,Bdv-Lovadina-Vacca:2018}
where the authors developed a (conforming) Virtual Element
Method for the Stokes and Navier--Stokes problems that
guarantees a divergence free velocity, a property that yields
advantages with respect to standard inf-sup stable schemes
(see for instance \cite{Linke3}). And, most importantly,
the proposed approach fitted quite naturally in the virtual
element setting, so that the ensuing element is not particularly
complicated to code or to handle.

Our aim is to further develop the idea in \cite{Stokes:divfree,Bdv-Lovadina-Vacca:2018},
also in order to get a deeper understanding of the underlying structure.
In \cite{falk-neilan:2013} the term ``Stokes exact complex'' was introduced;
in that paper the authors underline that, if a given velocity/pressure
FE scheme is associated to a discrete Stokes exact complex, than not only
the existence of an unique solution is guaranteed, but also the divergence-free
character of the discrete velocity. In addition, this allows to construct
an equivalent $\cc$ formulation of the problem in a potential-like variable.
This matter is one of  interest in the FEM community, see for instance
\cite{Linke1, Linke3}, also due to the difficulty in deriving exact Stokes
complexes for Finite Elements, that often yield quite ``cumbersome'' schemes. 

In the present paper, we unveil the underlying 2D Stokes complex
structure of the VEM in \cite{Stokes:divfree,Bdv-Lovadina-Vacca:2018}
by introducing a Virtual Element space $\Phi_h \subset H^2(\Omega)$ and proving that the triad $\{\Phi_h, \VV_h, Q_h\}$  (with $\VV_h$ and $Q_h$ velocity and pressure spaces of \cite{Bdv-Lovadina-Vacca:2018}) is an exact Stokes complex. Furthermore, we show the computability of the associated differential operators in terms of the adopted degrees of freedom (a key aspect in VEM discretizations) and we explore also a different discretization of the convective trilinear form.
As a byproduct of the above exact-sequence construction, we obtain a discrete $\cc$ formulation of the Navier--Stokes problem (set in the smaller space $\Phi_h$) that yields the same velocity as the original method (while the pressure needs to be recovered by solving a global rectangular system). For completeness, we also briefly present and compare a stream-function formulation approach, that is based on a direct discretization (with $C^1$ Virtual Elements) of the continuous stream function formulation of the problem. Some numerical tests are developed at the end of the paper, in order to show the performance of the methods, also comparing aspects such as condition number and size of the linear system. 
We note that a related study was developed in \cite{Antonietti-BeiraodaVeiga-Mora-Verani:20XX}, but only for the lowest order case without enhancements (that is, suitable for Stokes but not for Navier--Stokes).

The paper is organized as follows. In Section \ref{sec:2} we review the Navier--Stokes problem in strong and variational form, together with some basic theoretical facts. In Section \ref{sec:4} (after introducing some preliminaries and definitions in Section \ref{sec:3}) we review the Virtual scheme in \cite{Bdv-Lovadina-Vacca:2018}, propose a third option for the discretization of the convective term and extend the convergence results also to this case. In Section \ref{sec:5} we introduce the space $\Phi_h$ together with the associated degrees of freedom, prove the exact Stokes complex property and state the alternative curl formulation for the discrete problem. In Section \ref{sec:tests} we present a set of numerical tests, that also compare the proposed method with a direct $C^1$ discretization of the stream-function problem, briefly described in the Appendix, that is not associated to a Stokes complex.

Throughout the paper, we will follow the usual notation for Sobolev spaces
and norms \cite{Adams:1975}.
Hence, for an open bounded domain $\omega$,
the norms in the spaces $W^s_p(\omega)$ and $L^p(\omega)$ are denoted by
$\|{\cdot}\|_{W^s_p(\omega)}$ and $\|{\cdot}\|_{L^p(\omega)}$ respectively.
Norm and seminorm in $H^{s}(\omega)$ are denoted respectively by
$\|{\cdot}\|_{s,\omega}$ and $|{\cdot}|_{s,\omega}$,
while $(\cdot,\cdot)_{\omega}$ and $\|\cdot\|_{\omega}$ denote the $L^2$-inner product and the $L^2$-norm (the subscript $\omega$ may be omitted when $\omega$ is the whole computational
domain $\Omega$).
Moreover with a usual notation, the symbols $\gr$, $\Delta$ and $\boldsymbol{\nabla^2}$ denote the gradient, Laplacian and Hessian matrix for scalar functions, while 
$\dl $, $\Gr$, and $\dd$ denote the vector Laplacian,  the gradient operator
and the divergence  for vector fields.
Furthermore for a scalar function $\phi$ and a vector field $\vv := (v_1, \, v_2)$ we set
\begin{gather*}
\CC \, \phi := \left( \frac{\partial \phi}{\partial y}, \, -\frac{\partial \phi}{\partial x} \right) 
\qquad \text{and} \qquad
\cc \, \vv := \frac{\partial v_2}{\partial x}  -  \frac{\partial v_1}{\partial y} \,,
\\
\phi \times \vv := \phi \left( -v_2, \, v_1\right) \,.
\end{gather*}


\section{The Navier--Stokes equation}
\label{sec:2}

We consider the steady Navier--Stokes equation on a polygonal
simply connected domain $\Omega \subseteq \R^2$
(for more details, see for instance \cite{giraultbook}) 
%
\begin{equation}
\label{eq:ns primale}
\left\{
\begin{aligned}
& \text{ find $(\uu,p)$ such that}& &\\
& -  \nu \, \dl \uu + (\Gr \uu ) \,\uu -  \nabla p = \ff\qquad  & &\text{in $\Omega$,} \\
& \dd \, \uu = 0 \qquad & &\text{in $\Omega$,} \\
& \uu = 0  \qquad & &\text{on $\partial \Omega$,}
\end{aligned}
\right.
\end{equation}
where $\uu$, $p$ are the velocity and the pressure fields, respectively, 
$\nu \in {\mathbb R}$, $\nu > 0$ is the viscosity of the fluid and $\ff \in [L^2(\Omega)]^2$ represents 
the external force.
For sake of simplicity we here 
consider Dirichlet homogeneous boundary conditions, different boundary conditions can be treated as well. 
Problem \eqref{eq:ns primale} can be written in the equivalent \textit{rotational form}
\begin{equation}
\label{eq:ns primale rot}
\left\{
\begin{aligned}
& \mbox{ find $(\uu, P)$ such that}& &\\
& -  \nu \, \dl \uu + (\cc \, \uu) \times \uu -  \nabla P = \ff\qquad  & &\text{in $\Omega$,} \\
& \dd \, \uu = 0 \qquad & &\text{in $\Omega$,} \\
& \uu = 0  \qquad & &\text{on $\partial \Omega$.}
\end{aligned}
\right.
\end{equation}
Systems \eqref{eq:ns primale} and \eqref{eq:ns primale rot}
are equivalent in the sense that the velocity solutions $\uu$
coincide and the rotational pressure solution $P$ of Problem~\eqref{eq:ns primale rot},
the so-called \textit{Bernoulli pressure}, and the convective
pressure solution $p$ of Problem \eqref{eq:ns primale} are jointed by the relation
\begin{equation}
\label{eq:bernoulli}
P := p - \frac{1}{2} \,\uu \cdot \uu  +  \lambda \,,
\end{equation}
where, for the time being, $\lambda$ denotes a suitable constant.

Let us consider the spaces
\begin{equation*}
\VV:= \left[ H_0^1(\Omega) \right]^2, \qquad Q:= L^2_0(\Omega) = \left\{ q \in L^2(\Omega) \quad \text{s.t.} \quad \int_{\Omega} q \,{\rm d}\Omega = 0 \right\} 
\end{equation*}
endowed with natural norms
\begin{equation*}
\| \vv \|_{\VV} := | \vv|_{ H^1(\Omega)} \,, \qquad 
\|q\|_Q := \| q\|_{L^2(\Omega)}. 
\end{equation*}
Let us introduce the bilinear forms
\begin{gather}
\label{eq:forma a}
a(\cdot, \cdot) \colon \VV \times \VV \to \R,    \qquad 
a (\mathbf{u},  \mathbf{v}) := \int_{\Omega}  \, \boldsymbol{\nabla}   \mathbf{u} : \boldsymbol{\nabla}  \mathbf{v} \,{\rm d} \Omega, \qquad \text{for all $\mathbf{u},  \mathbf{v} \in \mathbf{V}$,}
\\
\label{eq:forma b}
b(\cdot, \cdot) \colon \mathbf{V} \times Q \to \R \qquad b(\mathbf{v}, q) :=  \int_{\Omega}q\, {\rm div} \,\mathbf{v} \,{\rm d}\Omega \qquad \text{for all $\mathbf{v} \in \mathbf{V}$, $q \in Q$,}
\end{gather}
and the trilinear forms
\begin{align}
\label{eq:forma cconv}
\cconv(\cdot; \, \cdot, \cdot) &\colon \VV \times \VV \times \VV \to \R 
&\quad 
&\cconv(\ww; \, \uu, \vv) :=  \int_{\Omega} ( \Gr \uu ) \, \ww \cdot \vv  \,{\rm d}\Omega \,,
\\
\label{eq:forma cskew}
\cskew(\cdot; \, \cdot, \cdot) &\colon \VV \times \VV \times \VV \to \R 
&\quad 
&\cskew(\ww; \, \uu, \vv) :=  
\frac{1}{2} \cconv(\ww; \, \uu, \vv) -
\frac{1}{2} \cconv(\ww; \, \vv, \uu) \,,
\\
\label{eq:forma crot}
\crot(\cdot; \, \cdot, \cdot) &\colon \VV \times \VV \times \VV \to \R 
&\quad 
&\crot(\ww; \, \uu, \vv) :=  \int_{\Omega} ( \cc \ww  \times \, \uu) \cdot \vv  \,{\rm d}\Omega \,.
\end{align}
Direct computations give that
\begin{gather}
\label{eq:cconv-cskew}
\cconv(\uu; \, \uu, \vv) = \cskew(\uu; \, \uu, \vv) \qquad \text{for all $\uu$, $\vv \in \VV$ with$\dd \, \uu = 0$,}
\\
\label{eq:cconv-crot}
\cconv(\uu; \, \uu, \vv) = \crot(\uu; \, \uu, \vv) + 
\frac{1}{2} \int_{\Omega} \gr (\uu \cdot \uu) \, \vv \, {\rm d}\Omega
\qquad \text{for all $\uu$, $\vv \in \VV$.}
\end{gather}
In the following we denote with $c(\cdot; \, \cdot, \cdot)$ one of the trilinear forms listed above.
Then a standard variational formulation of Problem~\eqref{eq:ns primale} is:
\begin{equation}
\label{eq:ns variazionale}
\left\{
\begin{aligned}
& \text{find $(\uu, p) \in \VV \times Q$, such that} \\
& \nu \, a(\uu, \vv) + c(\uu; \, \uu, \vv) + b(\vv, p) = (\ff, \vv) \qquad & \text{for all $\vv \in \VV$,} \\
&  b(\uu, q) = 0 \qquad & \text{for all $q \in Q$,}
\end{aligned}
\right.
\end{equation}
where
\[
(\ff, \vv) := \int_{\Omega} \ff \cdot \vv \, {\rm d} \Omega .
\]
From \eqref{eq:cconv-crot} it is clear that if
$c(\cdot; \, \cdot, \cdot) = \crot(\cdot; \, \cdot, \cdot)$
we recover instead the variational formulation of
system~\eqref{eq:ns primale rot} and the pressure solution
$p$ is actually the  Bernoulli pressure $P$ where $\lambda$
in \eqref{eq:bernoulli} is the mean value of $\frac{1}{2} \,\uu \cdot \uu$.

It is well known (see, for instance, \cite{giraultbook})
that in the diffusion dominated regime, i.e. under the assumption
\begin{equation}
\label{eq:ns condition}
\mathbf{(A0)} \qquad \gamma := \frac{\widehat{C} \, \|\ff\|_{H^{-1}}}{\nu^2} < 1
\end{equation}
where $\widehat{C}$ denotes the continuity constant of
$c(\cdot; \, \cdot, \cdot)$ with respect to the $\VV$-norm,
Problem \eqref{eq:ns variazionale} is well-posed and the
unique solution $(\mathbf{u}, p) \in \mathbf{V} \times Q$ satisfies
\begin{equation}
\label{eq:solution estimates}
\| \uu\|_{\VV} \leq \frac{\| \ff\|_{H^{-1}}}{\nu} \,.
\end{equation}
%
Finally, let us introduce the kernel of bilinear form $b(\cdot,\cdot)$,
\begin{equation}
\label{eq:kercontinuo}
\ZZ := \{ \vv \in \VV \quad \text{s.t.} \quad b(\vv, q) = 0 \quad \text{for all $q \in Q$} \}\,.
\end{equation}
Then, Problem~\eqref{eq:ns variazionale} can be formulated in the equivalent kernel form:
\begin{equation}
\label{eq:ns variazionale ker}
\left\{
\begin{aligned}
& \text{find $\uu \in \ZZ$, such that} \\
& \nu \, a(\uu, \vv) + c(\uu; \, \uu, \vv) = (\ff, \vv) \qquad & \text{for all $\vv \in \ZZ$.} 
\end{aligned}
\right.
\end{equation}
In this case, from \eqref{eq:cconv-cskew} and \eqref{eq:cconv-crot} it is straightforward to see that
\[
\cconv(\uu; \, \uu, \vv) = \cskew(\uu; \, \uu, \vv) = \crot(\uu; \, \uu, \vv)
\qquad \text{for all $\uu, \vv \in \ZZ$.}
\]
%


\subsection{Curl and Stream Formulation of the Navier--Stokes Equations}
\label{sub:2.2}

If $\Omega$ is a simply connected domain, 
a well known result (see  \cite{giraultbook} for the details) 
states that
a vector function $\vv \in \ZZ$ if and only if there exists a scalar potential function $\phi \in H^2(\Omega)$, called \textbf{stream function} such that 
\[
\vv = \CC \, \phi \,.
\] 
Clearly the function $\phi$ is defined up to a constant.
Let us consider the space
\begin{gather}
\label{eq:Phi}
\Phi := H^2_0(\Omega) = \left\{ \phi \in H^2(\Omega) \quad \text{s.t} \quad 
\phi = 0, \quad \frac{\partial \phi}{\partial n} = 0 
\quad \text{on $\partial \Omega$} \right\}
\end{gather} 
endowed with norm
\begin{equation*}
\| \phi\|_{\Phi} := |\Phi|_{H^2(\Omega)} \qquad \text{for all $\phi \in \Phi$.}
\end{equation*}
Then, Problem~\eqref{eq:ns variazionale ker} can be
formulated in the following $\CC$ \textit{formulation}:
\begin{equation}
\label{eq:ns variazionale curl}
\left\{
\begin{aligned}
& \text{find $\psi \in \Phi$, such that} \\
& \nu \, a(\CC \psi, \, \CC \phi) + c(\CC \psi; \, \CC \psi, \CC \phi) = (\ff, \CC \phi)
\qquad & \text{for all $\phi \in \Phi$.} 
\end{aligned}
\right.
\end{equation}


\noindent  
A different approach that makes use of the stream functions is the following.
Let $\psi$ be the stream function of the velocity  solution $\uu$ of \eqref{eq:ns primale},  i.e. $\uu = \CC \psi$.
Then applying the $\cc$ operator to the equation \eqref{eq:ns primale rot}, and 
using simple computations on the differential operators  
we obtain the equivalent following problem: 
\begin{equation}
\label{eq:ns stream}
\left\{
\begin{aligned}
& \mbox{ find $\psi$ such that}& &\\
&  \nu \, \Delta^2 \psi -  \CC\, \psi \cdot \gr(\Delta \psi) = \cc\, \ff\qquad  & &\text{in $\Omega$,} \\
& \psi = 0  \qquad & &\text{on $\partial \Omega$,} \\
& \frac{\partial \psi}{\partial n} = 0  \qquad & &\text{on $\partial \Omega$.}
\end{aligned}
\right.
\end{equation} 
This elliptic equation, can be reformulated in a variational
way as follows, obtaining the so-called \textit{stream formulation}
(we refer again to \cite{giraultbook}):
\begin{equation}
\label{eq:ns stream variazionale}
\left\{
\begin{aligned}
& \text{ find $\psi \in \Phi$ such that}& &\\
&   \nu \, \widetilde{a}(\psi, \, \phi) 
\, +  \, 
\widetilde{c}(\psi; \, \psi, \,  \phi)  = (\cc \, \ff, \, \phi) 
\qquad & &\text{for all  $\phi \in \Phi$,}
\end{aligned}
\right.
\end{equation} 
where
\begin{gather}
\label{eq:forma at}
\widetilde{a}(\cdot, \cdot) \colon \Phi \times \Phi \to \R,    \qquad 
\widetilde{a} (\psi, \, \phi) := \int_{\Omega}  \, \Delta \, \psi \,\Delta \, \phi\,{\rm d} \Omega, \qquad \text{for all $\psi,  \phi \in \Phi$,}
\\
\label{eq:forma crott}
\widetilde{c}(\cdot; \, \cdot, \cdot) \colon \Phi \times \Phi \times \Phi \to \R \quad 
\widetilde{c}(\zeta; \, \psi, \phi) :=  \int_{\Omega} \Delta \zeta \, \CC\psi \cdot \gr \phi   \,{\rm d}\Omega \quad \text{for all $\zeta, \psi, \phi \in \Phi$.}
\end{gather}
Since the formulations \eqref{eq:ns variazionale curl} and \eqref{eq:ns stream variazionale}  are equivalent to Problem \eqref{eq:ns variazionale ker} (in turn equivalent to Problem \eqref{eq:ns variazionale}), the well-posedness of $\CC$ and stream formulations  follows from assumption $\mathbf{(A0)}$. Moreover from \eqref{eq:solution estimates} follows the stability estimate
\begin{equation}
\label{eq:solution estimate stream}
\| \psi \|_{\Phi} \leq \frac{\|\ff\|_{H^{-1}}}{\nu} \,.
\end{equation}

\section{Definitions and preliminaries}
\label{sec:3}

In the present section we introduce some basic tools and notations useful in the construction and theoretical analysis of Virtual Element Methods.
Let $\set{\Omega_h}_h$ be a sequence of decompositions of $\Omega$ into general polygonal elements $E$ with
\[
 h_E := {\rm diameter}(E) , \quad
h := \sup_{E \in \Omega_h} h_E .
\]
We suppose that for all $h$, each element $E$ in $\Omega_h$ fulfils the following assumptions:
\begin{description}
\item [$\mathbf{(A1)}$] $E$ is star-shaped with respect to a ball $B_E$ of radius $ \geq\, \rho \, h_E$, 
\item [$\mathbf{(A2)}$] the distance between any two vertexes of $E$ is $\geq \rho \, h_E$, 
\end{description}
where $\rho$ is a uniform positive constant. We remark that the hypotheses above, though not too restrictive in many practical cases, 
can be further relaxed, as investigated in ~\cite{2016stability}. 
For any $E \in \Omega_h$, using standard VEM notations, for $n \in \N$  let us introduce the spaces:
\begin{itemize}
\item $\Pk_n(E)$ the set of polynomials on $E$ of degree $\leq n$  (with the extended notation $\Pk_{-1}(E)=\emptyset$),
\item $\B_n(\partial E) := \{v \in C^0(\partial E) \quad \text{s.t} \quad v_{|e} \in \Pk_n(e) \quad \text{for all edge $e \subset \partial E$} \}$.
%
\end{itemize}
Note that
\begin{equation}
\label{eq:pkgkgkp}
[\Pk_n(E)]^2 = \nabla(\Pk_{n+1}(E))  \oplus \mathbf{x}^{\perp} \,\Pk_{n-1}(E)   
\end{equation}
 with $\mathbf{x}^{\perp}:= (x_2, -x_1)$.
%

\begin{remark}
\label{fact1}
Note that \eqref{eq:pkgkgkp} implies that
the operator $\cc$ is an isomorphism from $\mathbf{x}^{\perp} \, \Pk_{n-1}(E)$ to the whole $\Pk_{n-1}(E)$, i.e. for any $q_{n-1} \in \Pk_{n-1}(E)$ there exists a unique $p_{n-1} \in \Pk_{n-1}(E)$ such that
\[
q_{n-1} = \cc (\mathbf{x}^{\perp} \,p_{n-1} ) \,.
\]
\end{remark} 
%
%
\noindent
We also have that
\begin{equation}
\label{eq:dimensions1}
\dim \left(\Pk_n(E) \right) = \frac{n(n+1)}{2} \,,
\qquad
\dim \left(\B_n(\partial E)\right) = n_E \, n
\end{equation}
where $n_E$ is the number of edges (or the number of vertexes) of the polygon $E$.
%

One core idea in the VEM construction is to define suitable (computable) polynomial projections. 
For any $n \in \N$ and $E \in \Omega_h$ we introduce the following polynomial projections:
\begin{itemize}

\item the $\boldsymbol{L^2}$\textbf{-projection} $\Pi_n^{0, E} \colon L^2(E) \to \Pk_n(E)$, given by
\begin{equation}
\label{eq:P0_k^E}
\int_Eq_n (v - \, {\Pi}_{n}^{0, E}  v) \, {\rm d} E = 0 \qquad  \text{for all $v \in L^2(E)$  and for all $q_n \in \Pk_n(E)$,} 
\end{equation} 
with obvious extension for vector functions $\Pi_n^{0, E} \colon [L^2(E)]^2 \to [\Pk_n(E)]^2$, and tensor functions 
$\boldsymbol{\Pi}_{n}^{0, E} \colon [L^2(E)]^{2 \times 2} \to [\Pk_{n}(E)]^{2 \times 2}$,

\item the $\boldsymbol{H^1}$\textbf{-seminorm projection} ${\Pi}_{n}^{\nabla,E} \colon H^1(E) \to \Pk_n(E)$, defined by 
\begin{equation}
\label{eq:Pn_k^E}
\left\{
\begin{aligned}
& \int_E \gr  \,q_n \cdot \gr ( v - \, {\Pi}_{n}^{\nabla,E}   v) \, {\rm d} E = 0 \qquad  \text{for all $v \in H^1(E)$ and for all $q_n \in \Pk_n(E)$,} \\
& \Pi_0^{0,E}(v - \,  {\Pi}_{n}^{\nabla, E}  v) = 0 \, ,
\end{aligned}
\right.
\end{equation} 
with obvious extension for vector functions $\Pi_n^{\nabla, E} \colon [H^1(E)]^2 \to [\Pk_n(E)]^2$,

\item the $\boldsymbol{H^2}$\textbf{-seminorm projection} ${\Pi}_{n}^{\nabla^2,E} \colon H^2(E) \to \Pk_n(E)$, defined by 
\begin{equation}
\label{eq:Pn2_k^E}
\left\{
\begin{aligned}
& \int_E \boldsymbol{\nabla^2} \,q_n : \boldsymbol{\nabla^2} (v- \, {\Pi}_{n}^{\nabla^2,E}   v) \, {\rm d} E = 0 \qquad  \text{for all $v \in H^2(E)$ and for all $q_n \in \Pk_n(E)$,} \\
& \int_{\partial E}v - \,  {\Pi}_{n}^{\nabla^2, E}  v  \, {\rm d}s= 0 \, 
\qquad \text{and} \qquad
\int_{\partial E} \frac{\partial}{\partial n} \left(v - \,  {\Pi}_{n}^{\nabla^2, E}  v\right)  \, {\rm d}s= 0 \,.
\end{aligned}
\right.
\end{equation}

\end{itemize}
Finally, let us recall a classical approximation result for $\Pk_n$
polynomials on star-shaped domains, see for instance \cite{brennerscott}:
\begin{lemma}
\label{fact2}
Under the assumptions $\mathbf{(A1)}$ and $\mathbf{(A2)}$,
for any $v \in H^{s+1}(E)$, with $s \geq 0$, it holds
\begin{gather*}
\|v - \Pi_n^{0, E} v\|_{m, E} \leq C \, h^{s +1 - m}_E \, |v|_{s+1, E}
\qquad \text{for $m \leq s + 1 \leq n+1$,}
\\
\|v - \Pi_n^{\nabla, E} v \|_{m, E} \leq C \, h^{s + 1- m}_E \, |v|_{s+1, E}
\qquad \text{for $m \leq s + 1 \leq n+1$, }
\end{gather*}
with $C$ depending only on $n$ and the shape  constant $\rho$ in assumptions $\mathbf{(A1)}$
and $\mathbf{(A2)}$.
\end{lemma} 
In the following the symbol $C$ will indicate a generic positive constant,
independent of the mesh size $h$, the viscosity $\nu$ and the constant
$\gamma$ appearing in $\mathbf{(A0)}$, but which may depend on $\Omega$,
the integer $k$ (representing the ``polynomial'' order of the method) and
on the shape constant $\rho$ in assumptions $\mathbf{(A1)}$
and $\mathbf{(A2)}$. 
Furthermore, $C$ may vary at each occurrence.

\section{Virtual elements velocity-pressure formulation}
\label{sec:4}
 
In the present section we outline a  short overview of the Virtual
Element discretization of Navier--Stokes Problem \eqref{eq:ns variazionale}. 
We will make use of various tools from the virtual element technology,
that will be described briefly; we refer the interested
reader to the papers
\cite{Antonietti-BeiraodaVeiga-Mora-Verani:20XX,Stokes:divfree,vacca:2018, Bdv-Lovadina-Vacca:2018}.
We recall that in \cite{Stokes:divfree}  a new family of Virtual Elements
for the Stokes Equation has been introduced.
The core idea is to define suitable Virtual space of velocities,
associated to a Stokes-like variational problem on each element,
such that the discrete velocity kernel is pointwise divergence-free. 
In \cite{vacca:2018} has been presented an enhanced Virtual space
to be used in place of the original one, that, taking the inspiration from \cite{Ahmed-et-al:2013},
allows the explicit knowledge of the $L^2$-projection onto
the  polynomial space $\Pk_k$ (being $k$ the order of the method). 
In \cite{Bdv-Lovadina-Vacca:2018} a Virtual Element scheme for
the Navier--Stokes equation in classical velocity-pressure
formulation has been proposed. In the following we give some
basic tools and a brief overview of such scheme.
We focus particularly on the virtual element discretization
of Navier--Stokes equation in rotation form \eqref{eq:ns primale rot}
related to the trilinear form $\crot(\cdot; \, \cdot, \cdot)$
defined in \eqref{eq:forma crott} that was not treated in \cite{Bdv-Lovadina-Vacca:2018}. 
Specifically for the resulting discrete scheme we develop
the convergence analysis for both the Bernoulli and the related convective pressure.

\subsection{Virtual elements spaces}
\label{sub:4.1}

Let $k \geq 2$ the polynomial degree of accuracy of the method. We consider  on each element $E \in \Omega_h$ the (enlarged) finite dimensional local virtual space
\begin{multline*}
\mathbf{U}_h^E := \biggl\{  
\mathbf{v} \in [H^1(E)]^2 \quad \text{s.t} \quad \mathbf{v}_{|{\partial E}} \in [\B_k(\partial E)]^2 \, , \biggr.
\\
\left.
\biggl\{
\begin{aligned}
& - \boldsymbol{\Delta}    \mathbf{v}  -  \nabla s \in \mathbf{x}^{\perp}\, \Pk_{k-1}(E),  \\
& {\rm div} \, \mathbf{v} \in \Pk_{k-1}(E),
\end{aligned}
\biggr. \qquad \text{ for some $s \in  L^2(E) \setminus \R$}
\quad \right\}.
\end{multline*}
Now, we define the Virtual Element space $\mathbf{V}_h^E$
as the restriction  of $\mathbf{U}_h^E$ given by (cf. \cite{Bdv-Lovadina-Vacca:2018}):
\begin{equation}
\label{eq:V_h^E}
\mathbf{V}_h^E := \left\{ \mathbf{v} \in \mathbf{U}_h^E \quad \text{s.t.} \quad   \left(\mathbf{v} - \Pi^{\nabla,E}_k \mathbf{v}, \, \mathbf{x}^{\perp}\, \widehat{p}_{k-1} \right)_{E} = 0 \quad \text{for all $\widehat{p}_{k-1} \in  \widehat{\Pk}_{k-1 \setminus k-3}(E)$} \right\} ,
\end{equation}
where the symbol $\widehat{\Pk}_{k-1 \setminus k-3}(E) := \Pk_{k-1}(E) \setminus \Pk_{k-3}(E)$
denotes the polynomials in $\Pk_{k-1}(E)$ that are $L^2$-orthogonal
to all polynomials of $\Pk_{k-3}(E)$.
%
%

We here summarize the main properties of the virtual
space $\VV_h^E$  (we refer  \cite{vacca:2018,Bdv-Lovadina-Vacca:2018} for a deeper analysis):
\begin{itemize}
\item \textbf{dimension:}
 the dimension of $\VV_h^E$  is
 \begin{equation}
\label{eq:dimensione V_h^E}
\begin{split}
\dim\left( \mathbf{V}_h^E \right) 
&= 2n_E \, k + \frac{(k-1)(k-2)}{2}  + \frac{(k+1)k}{2} - 1
\end{split}
\end{equation}
where $n_E$ is the number of vertexes of $E$;
\item \textbf{degrees of freedom:}
the following linear operators $\mathbf{D_V}$, split into four subsets (see Figure \ref{fig:dofsloc}) constitute a set of DoFs for $\VV_h^E$:
\begin{itemize}
\item $\mathbf{D_V1}$:  the values of $\mathbf{v}$ at the vertexes of the polygon $E$,
\item $\mathbf{D_V2}$: the values of $\mathbf{v}$ at $k-1$ distinct points of every edge $e \in \partial E$,
\item $\mathbf{D_V3}$: the moments of $\mathbf{v}$
\[
\int_E \mathbf{v} \cdot \mathbf{x}^{\perp}\, p_{k-3} \, {\rm d}E \qquad \text{for all $p_{k-3} \in \Pk_{k-3}(E)$,}
\]
\item $\mathbf{D_V4}$: the moments of ${\rm div} \,\mathbf{v}$ 
\[
\int_E ({\rm div} \,\mathbf{v}) \, q_{k-1} \, {\rm d}E \qquad \text{for all $q_{k-1} \in \Pk_{k-1}(E) \setminus \R$;}
\] 
\end{itemize}
\item \textbf{projections:}
the DoFs $\mathbf{D_V}$ allow us to compute exactly (c.f. \eqref{eq:Pn_k^E} and \eqref{eq:P0_k^E})
\[
\PN \colon \VV_h^E \to [\Pk_k(E)]^2, \qquad
\P0 \colon \VV_h^E \to [\Pk_k(E)]^2, \qquad
\PP0 \colon \Gr(\VV_h^E) \to [\Pk_{k-1}(E)]^{2 \times 2},
\]
in the sense that, given any $\vv_h \in \VV_h^E$, we are able to compute the polynomials
$\PN \vv_h$, $\P0 \vv_h$ and $\PP0\nabla\vv_h$ only using, as unique information, the degree of freedom values $\mathbf{D_V}$ of $\vv_h$.
\end{itemize}
\begin{figure}[!h]
\center{
\includegraphics[scale=0.25]{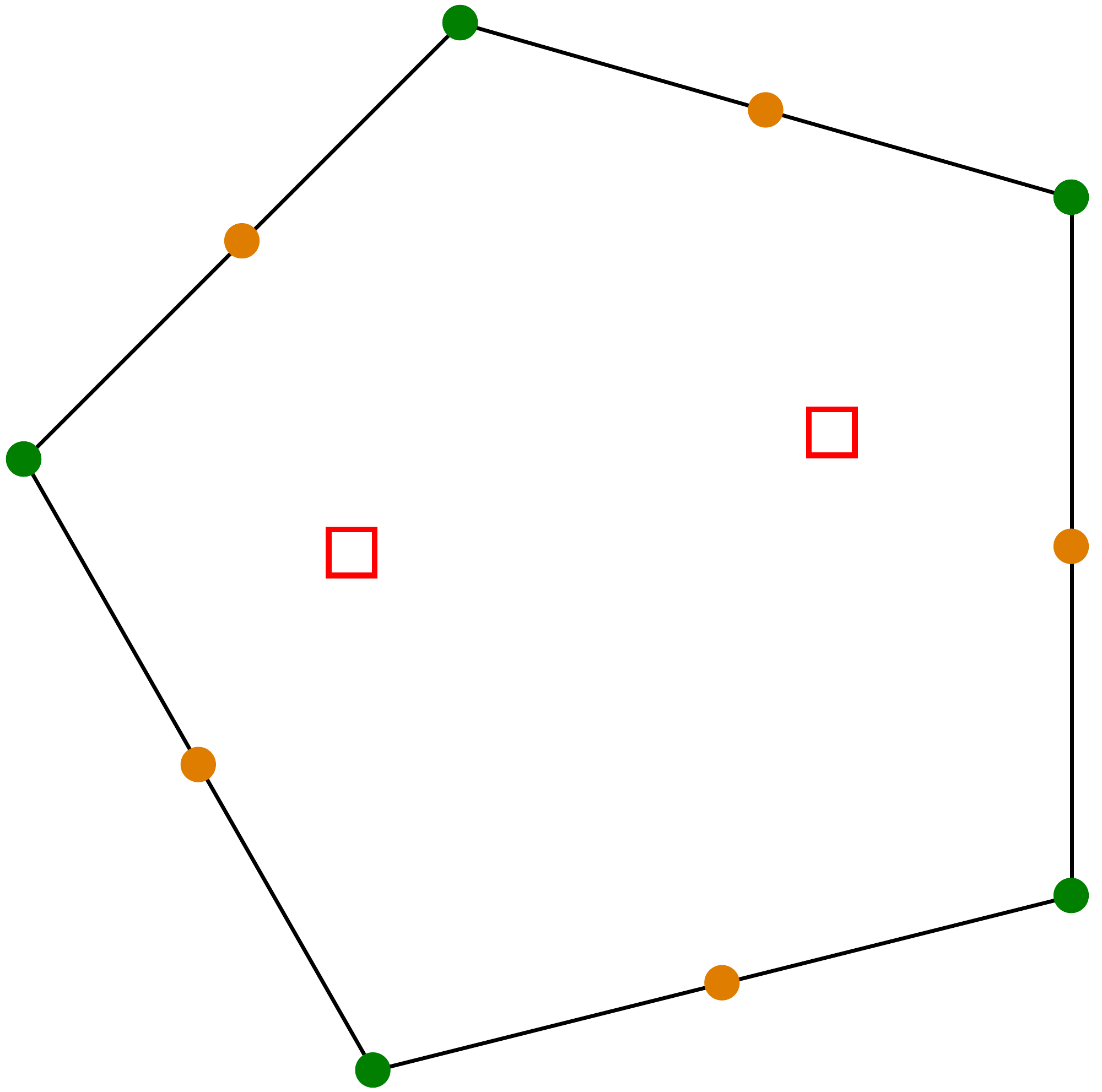} 
\qquad \qquad \qquad
\includegraphics[scale=0.25]{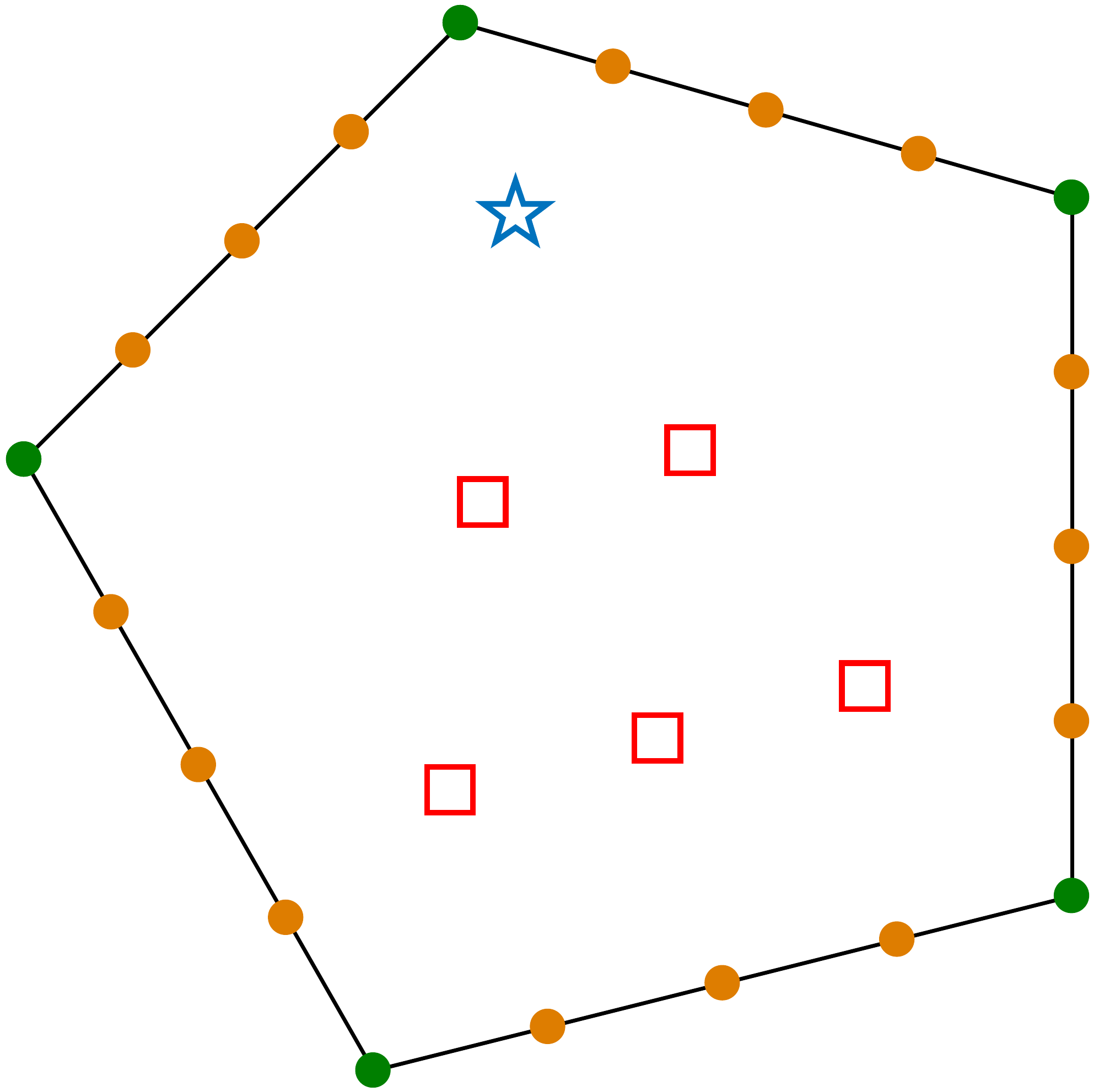}
\caption{DoFs for $k=2$ (left), $k=3$ (right). We denote $\mathbf{D_V1}$ with green dots, $\mathbf{D_V2}$ with orange dots, $\mathbf{D_V3}$ with blue stars, $\mathbf{D_V4}$ with red squares.}
\label{fig:dofsloc}
}
\end{figure}
The global virtual element space is obtained as usual by combining the local spaces $\VV_h^E$
accordingly to the local degrees of freedom, as in standard  finite elements and considering the homogeneous boundary conditions. We obtain the global space
\begin{equation}
\label{eq:V_h}
\mathbf{V}_h := \{ \mathbf{v} \in [H^1_0(\Omega)]^2  \quad \text{s.t} \quad \mathbf{v}_{|E} \in \mathbf{V}_h^E  \quad \text{for all $E \in \Omega_h$} \} \,.
\end{equation} 
The space $\VV_h$ has an optimal interpolation order of
accuracy with respect to $k$ (see Theorem 4.1 in \cite{Bdv-Lovadina-Vacca:2018}).
\begin{theorem}
\label{thm:interpolante}
Let $\vv \in H^{s+1}(\Omega) \cap \VV$, for $0<s \le k$.
Then, there exists $\vv_I \in \VV_h$ such that
\[
\| \vv - \vv_I \|_0 + h \, \| \vv - \vv_I \|_{\VV} \leq C \, h^{s+1} \, | \vv |_{s+1},
\]
where the constant $C$ depends only on the degree $k$
and the shape regularity constant $\rho$ (see assumptions $\mathbf{(A1)}$
and $\mathbf{(A2)}$ of Section~\ref{sec:3}).
\end{theorem}
The pressure space is simply given by the piecewise polynomial functions
\begin{equation}
\label{eq:Q_h}
Q_h := \{ q \in L_0^2(\Omega) \quad \text{s.t.} \quad q_{|E} \in
\Pk_{k-1}(E) \quad \text{for all $E \in \Omega_h$}\},
\end{equation}
with the obvious associated set of global degrees of freedom:
\begin{itemize}
\item $\mathbf{D_Q}$: the moments up to order $k-1$ of $q$, i.e.
\[
\int_E q \, p_{k-1} \, {\rm d}E \qquad \text{for all $p_{k-1} \in \Pk_{k-1}(E)$
and for all $E \in \Omega_h$.}
\]
\end{itemize}
%
%
%
%
A crucial observation is that, by definitions \eqref{eq:V_h} and \eqref{eq:Q_h}, it holds 
\begin{equation}\label{eq:divfree}
{\rm div}\, \mathbf{V}_h\subseteq Q_h \,.
\end{equation}
Therefore the discrete kernel
\begin{equation}
\label{eq:Zh}
\mathbf{Z}_h := \{ \mathbf{v}_h \in \mathbf{V}_h \quad \text{s.t.} \quad b(\mathbf{v}_h, q_h) = 0 \quad \text{for all $q_h \in Q_h$}\},
\end{equation}
is a subspace of the continuous kernel $\ZZ$ defined in \eqref{eq:kercontinuo}, i.e.
\begin{equation}
\label{eq:inclusion}
\ZZ_h \subseteq \ZZ \,.
\end{equation}
This leads to a series of important advantages, as explored in 
\cite{Linke2, Linke3, Bdv-Lovadina-Vacca:2018,vacca:2018}.
Finally, we remark that the kernel $\ZZ_h$ is obtained
by gluing the local kernels explicitly defined by
\begin{multline}
\label{eq:Z_h^E}
\mathbf{Z}_h^E := \biggl\{  
\mathbf{v} \in [H^1(E)]^2 \quad \text{s.t} \quad   
\left(\mathbf{v} - \Pi^{\nabla,E}_k \mathbf{v}, \, \mathbf{x}^{\perp} \, \widehat{p}_{k-1} \right)_{E} = 0 \quad \text{for all $\widehat{p}_{k-1} \in  \widehat{\Pk}_{k-1 \setminus k-3}(E)$}
\biggr.
\\
\left.
\mathbf{v}_{|{\partial E}} \in [\B_k(\partial E)]^2 \, , 
\quad 
\biggl\{
\begin{aligned}
& - \boldsymbol{\Delta}    \mathbf{v}  -  \nabla s \in \mathbf{x}^{\perp} \, \Pk_{k-1},  \\
& {\rm div} \, \mathbf{v} = 0,
\end{aligned}
\biggr. \qquad \text{ for some $s \in L^2(E) \setminus \R$}
\quad \right\} \,.
\end{multline}
The dimension of $\ZZ_h^E$  is
\begin{equation}
\label{eq:dimensione Z_h^E}
\begin{split}
\dim\left( \mathbf{Z}_h^E \right) 
&= 2n_E k + \frac{(k-1)(k-2)}{2}  - 1.
\end{split}
\end{equation}
where $n_E$ is the number of vertexes of $E$.

\subsection{Discrete bilinear forms and load term approximation}
\label{sub:4.2}

In this subsection we briefly describe the
construction of a discrete version of the bilinear form
$a(\cdot, \cdot)$  given in \eqref{eq:forma a}
and trilinear forms $c(\cdot; \cdot, \cdot)$
(cf. \eqref{eq:forma cconv}, \eqref{eq:forma cskew}, \eqref{eq:forma crot}).
We can follow in a rather slavish way the procedure initially introduced in \cite{volley} for the laplace problem and further developed in \cite{Stokes:divfree, vacca:2018, Bdv-Lovadina-Vacca:2018} for flow problems.
First, we decompose into local contributions the bilinear
form $a(\cdot, \cdot)$ and the trilinear forms $c(\cdot; \cdot, \cdot)$ by considering:
\begin{equation*}
a (\mathbf{u},  \mathbf{v}) =: \sum_{E \in \Omega_h} a^E (\mathbf{u},  \mathbf{v}) \,,
\qquad
c(\ww; \, \uu, \vv) =: \sum_{E \in \Omega_h} c^E(\ww; \, \uu, \vv) \,
\qquad 
\text{for all $\ww, \uu, \vv \in \VV$.}
\end{equation*}
Therefore, following a standard procedure in the VEM framework, we define a computable discrete local bilinear form
\begin{equation}
\label{eq:a_h^E} 
a_h^E(\cdot, \cdot) \colon \mathbf{V}_h^E \times \mathbf{V}_h^E \to \R
\end{equation}
approximating the continuous form $a^E(\cdot, \cdot)$, and defined by
\begin{equation}
\label{eq:a_h^E def}
a_h^E(\mathbf{u}_h, \mathbf{v}_h) := a^E \left(\PN \mathbf{u}_h, \, \PN \mathbf{v}_h \right) + \mathcal{S}^E \left((I - \PN) \mathbf{u}_h, \, (I -\PN) \mathbf{v}_h \right)
\end{equation}
for all $\mathbf{u}_h, \mathbf{v}_h \in \mathbf{V}_h^E$, where the (symmetric) stabilizing bilinear form $\mathcal{S}^E \colon \mathbf{V}_h^E \times \mathbf{V}_h^E \to \R$ satisfies 
\begin{equation}
\label{eq:S^E}
\alpha_* a^E(\mathbf{v}_h, \mathbf{v}_h) \leq  \mathcal{S}^E(\mathbf{v}_h,
\mathbf{v}_h) \leq \alpha^* a^E(\mathbf{v}_h, \mathbf{v}_h) \qquad
\text{$\forall \, \mathbf{v}_h \in \mathbf{V}_h^E$,   ${\Pi}_{k}^{\nabla ,E} \mathbf{v}_h= \mathbf{0}$}
\end{equation}
with $\alpha_*$ and $\alpha^*$  positive  constants independent of the element $E$.
For instance, a standard choice is 
$\mathcal{S}^E(\mathbf{u}_h, \mathbf{v}_h) = \sum_{i=1}^{N_{DoFs}}\mathbf{D_V}_i(\mathbf{u}_h) \, \mathbf{D_V}_i(\mathbf{v}_h)$
where $\mathbf{D_V}_i(\mathbf{u}_h)$ denotes the $i$-th  DoFs value of 
$\uu_h$, opportunely scaled.
It is straightforward to check that the definition of $\Pi^{\nabla, E}_{k}$ projection ~\eqref{eq:Pn_k^E} and properties~\eqref{eq:S^E} imply 
that the discrete form $a_h^E(\cdot, \cdot)$
satisfies the consistency and stability properties.
The global approximated bilinear form $a_h(\cdot, \cdot) \colon \mathbf{V}_h \times \mathbf{V}_h \to \R$ is defined by simply summing the local contributions:
\begin{equation}
\label{eq:a_h}
a_h(\mathbf{u}_h, \mathbf{v}_h) := \sum_{E \in \Omega_h}  a_h^E(\mathbf{u}_h, \mathbf{v}_h) \qquad \text{for all $\mathbf{u}_h, \mathbf{v}_h \in \mathbf{V}_h$.}
\end{equation}
We now define discrete versions of the forms  $c(\cdot; \,\cdot, \cdot)$. 
Referring to \eqref{eq:forma cconv},
\eqref{eq:forma cskew}, \eqref{eq:forma crot} 
we set
for all $\ww_h, \uu_h, \vv_h \in \VV_h^E$:
\begin{align}
\label{eq:cconv_h^E}
\cconvh^E(\ww_h; \, \uu_h, \vv_h) &:= \int_E \left[ \left(\PP0 \, \Gr \uu_h  \right)  \left(\P0 \ww_h \right) \right] \cdot \P0 \vv_h \, {\rm d}E 
\\
\label{eq:cskew_h^E}
\cskewh^E(\ww_h; \, \uu_h, \vv_h) &:= \frac{1}{2}\cconvh^E(\ww; \, \uu, \vv) - \frac{1}{2}\cconvh^E(\ww; \, \vv, \uu)
\\
\label{eq:crot_h^E}
\croth^E(\ww_h; \, \uu_h, \vv_h) &:= \int_E \left[ \left(\Pi_{k-1}^{0,E} \, \cc \ww_h  \right) \times \left(\P0 \uu_h \right) \right] \cdot \P0 \vv_h \, {\rm d}E 
\end{align}
and note that all quantities in the previous formulas are computable. 
Let $c_h^E(\cdot; \, \cdot, \cdot)$ be one of the discrete trilinear forms listed above.
As usual we define the global approximated trilinear form by adding the local contributions:
\begin{equation}
\label{eq:c_h}
c_h(\ww_h; \, \uu_h, \vv_h) := \sum_{E \in \Omega_h}  c_h^E(\ww_h; \, \uu_h, \vv_h), \qquad \text{for all $\ww_h, \uu_h, \vv_h \in \VV_h$.}
\end{equation}
The forms $c_h(\cdot; \, \cdot, \cdot)$ are immediately extendable to the whole $\VV$ (simply apply the same definition for any $\ww, \uu, \vv \in \VV$). 
Moreover, the trilinear forms $c_h(\cdot; \, \cdot, \cdot)$ are continuous on $\VV$, i.e. 
there exists a uniform bounded constant $\widehat{C}_h$ such that
\begin{equation}
\label{eq:CC}
|c_h(\ww; \, \uu, \vv)| \leq 
\widehat{C}_h \, \|\ww\|_{\VV} \|\uu\|_{\VV} \|\vv\|_{\VV}
\qquad \text{for all $\ww, \uu, \vv \in \VV$.}
\end{equation}
The proof of the continuity for the trilinear forms $\cconvh(\cdot; \, \cdot, \cdot)$ and
$\cskewh(\cdot; \, \cdot, \cdot)$ can be found in Proposition 3.3 in \cite{Bdv-Lovadina-Vacca:2018}. Analogous techniques can be used to prove  the continuity of the trilinear form $\croth(\cdot; \, \cdot, \cdot)$.
For what concerns $b(\cdot, \cdot)$, as noticed in \cite{Stokes:divfree} we do not need to introduce any approximation of the bilinear form, since it can be exactly computed by the DoFs $\mathbf{D_{\VV}}$.

The last step consists in constructing a computable
approximation of the right-hand side $(\mathbf{f}, \, \mathbf{v})$
in \eqref{eq:ns variazionale}.  We define the approximated load term $\mathbf{f}_h$ as 
\begin{equation}
\label{eq:f_h}
\mathbf{f}_h|_E := \Pi_{k}^{0,E} \mathbf{f} \qquad \text{for all $E \in \Omega_h$,}
\end{equation}
and consider:
\begin{equation}
\label{eq:right}
(\mathbf{f}_h, \mathbf{v}_h)  = \sum_{E \in \Omega_h} \int_E \mathbf{f}_h \cdot \mathbf{v}_h \, {\rm d}E = \sum_{E \in \Omega_h} \int_E \Pi_{k}^{0,E} \mathbf{f} \cdot \mathbf{v}_h \, {\rm d}E = \sum_{E \in \Omega_h} \int_E \mathbf{f} \cdot \Pi_{k}^{0,E}  \mathbf{v}_h \, {\rm d}E \,.
\end{equation}


\subsection{The discrete problem}\label{sec:discrete}
\label{sub:4.3}

Referring to~\eqref{eq:V_h}, \eqref{eq:Q_h},  ~\eqref{eq:a_h},  ~\eqref{eq:c_h}, ~\eqref{eq:forma b}
and \eqref{eq:right}, the virtual element approximation of the
Navier--Stokes equation in the velocity-pressure formulation is given by:
\begin{equation}
\label{eq:ns virtual}
\left\{
\begin{aligned}
& \text{find $(\uu_h, p_h) \in \VV_h \times Q_h$, such that} \\
& \nu \, a_h(\uu_h, \vv_h) + c_h(\uu_h; \,  \uu_h, \vv_h) + b(\vv_h, p_h) = (\ff_h, \vv_h) \qquad & \text{for all $\vv_h \in \VV_h$,} \\
&  b(\uu_h, q_h) = 0 \qquad & \text{for all $q_h \in Q_h$,}
\end{aligned}
\right.
\end{equation}
with $c_h(\cdot; \, \cdot, \cdot)$ given by \eqref{eq:cconv_h^E},
\eqref{eq:cskew_h^E} or \eqref{eq:crot_h^E}. 
Whenever the choice \eqref{eq:crot_h^E} is adopted, the pressure
output in \eqref{eq:ns virtual} approximates the Bernoulli
pressure $P$ in \eqref{eq:bernoulli} instead of the convective pressure $p$.
Recalling the kernel inclusion \eqref{eq:inclusion},
Problem~\eqref{eq:ns virtual} can be also formulated in the equivalent kernel form
\begin{equation}
\label{eq:nsvirtual ker}
\left\{
\begin{aligned}
& \text{find $\uu_h \in \ZZ_h$, such that} \\
& \nu \, a_h(\uu_h, \vv_h) + c_h(\uu_h; \, \uu_h, \vv_h) = (\ff_h, \vv_h) \qquad & \text{for all $\vv_h \in \ZZ_h$.} 
\end{aligned}
\right.
\end{equation}
The well-posedness of the discrete problems can be
stated in the following theorem (cf. \cite{Bdv-Lovadina-Vacca:2018}).
\begin{theorem}
\label{thm:well}
Under the assumption
\begin{equation}
\label{eq:ns virtual condition}
\mathbf{(A0)_h} 
\qquad 
\gamma_h := \frac{\widehat{C}_h \, \|\ff_h\|_{H^{-1}}}{\alpha_*^2 \, \nu^2}  < 1 \,
\end{equation}
Problem \eqref{eq:ns virtual} has a unique solution $(\uu_h, p_h) \in \VV_h \times Q_h$ such that
\begin{equation}
\label{eq:solution virtual estimates}
\| \uu_h\|_{\VV} \leq \frac{\| \ff_h\|_{H^{-1}}}{\alpha_* \, \nu} \,.
\end{equation}
\end{theorem}
\noindent
We have the following approximation results (see Theorem 4.6
and Remark 4.2 in \cite{Bdv-Lovadina-Vacca:2018}
for the choices \eqref{eq:cconv_h^E} and \eqref{eq:cskew_h^E}). 
\begin{theorem}
\label{thm:u}
Under the assumptions  $\mathbf{(A0)}$ and $\mathbf{(A0)_h}$, 
let $(\uu, p) \in \VV \times Q$ be the solution of Problem \eqref{eq:ns variazionale} and $(\uu_h, p_h) \in \VV_h \times Q_h$ be the solution of Problem \eqref{eq:ns virtual}. 
Assuming moreover $\uu, \ff \in [H^{s+1}(\Omega)]^2$ and $p \in H^s(\Omega)$, $0 < s \le k$, then 
\begin{gather}
\label{eq:thm:u}
\| \uu - \uu_h \|_{\VV} \leq \, h^{s} \, \mathcal{F}(\uu; \, \nu, \gamma, \gamma_h) + \, h^{s+2} \, \mathcal{H}(\ff; \nu, \gamma_h)
\\
\label{eq:p-est}
\|p  - p_h\|_Q \leq C \, h^{s} \, |p|_{s} +  h^s \, \mathcal{K}(\uu; \nu, \gamma, \gamma_h) + C \, h^{s+2} \, |\ff|_{s+1} 
\end{gather}
for a suitable functions 
$\mathcal{F}$, 
$\mathcal{H}$,
$\mathcal{K}$ 
independent of $h$. 
\end{theorem}

Following the same steps as in \cite{Bdv-Lovadina-Vacca:2018}, Theorem \ref{thm:u} easily extends also to the choice \eqref{eq:crot_h^E}. In such case 
we preliminary observe that if the velocity solution $\uu \in [H^{s+1}(\Omega)]^2$ and the convective pressure $p \in H^s(\Omega)$ then the Bernoulli pressure $P$  is in $H^s(\Omega)$.
As a matter of fact, recalling \eqref{eq:bernoulli} by the H\"older inequality and Sobolev embedding $H^{s+1}(\Omega) \subset W^{s}_4(\Omega)$, we recover
\[
\|P\|_s \leq
\|p\|_s + \frac{1}{2} \, \|\uu \cdot \uu\|_s + \|\lambda\|_s \leq
\|p\|_s + \frac{1}{2} \, \|\uu \|_{W^s_4}^2 + \|\lambda\|_s \leq
 \|p\|_s + \frac{1}{2} \, \|\uu\|^2_{s+1} + |\lambda| |\Omega|^{1/2}\,.
\]
Now let $(\uu_h,\, P_h)$ be the solution of the virtual problem \eqref{eq:ns virtual} with the trilinear form \eqref{eq:crot_h^E} and $(\uu, \, P)$ be the solution of the Navier--Stokes equation \eqref{eq:ns primale rot}. Then \eqref{eq:p-est} is substituted by
\begin{equation}
\label{eq:P-est}
\|P  - P_h\|_Q \leq C \, h^{s} \, |P|_{s} +  h^s \, \mathcal{K}(\uu; \nu, \gamma, \gamma_h) + C \, h^{s+2} \, |\ff|_{s+1} \,.
\end{equation}
In such case the following computable approximation $p_h$ of the convective pressure $p$ is available:
\begin{equation}
\label{eq:ph_convective}
{p_h}_{|E} := {P_h}_{|E} +
\frac{1}{2} \,\Pi_k^{0, E} \uu_h \cdot  \Pi_k^{0, E} \uu_h - 
\lambda_h
\end{equation}
where $\lambda_h$ is the mean value of the piecewise polynomial function $\frac{1}{2} \,\Pi_k^{0, E} \uu_h \cdot  \Pi_k^{0, E} \uu_h$.
The optimal order of accuracy for the convective pressure  can established as follows.
Definitions \eqref{eq:bernoulli} (taking $\lambda$ as the mean value of 
$\frac{1}{2} \, \uu \cdot \uu$)
and \eqref{eq:ph_convective} easily imply
\begin{equation}
\label{eq:bernoulliest1}
\begin{split}
\|p - p_h\|_Q &\leq 	\|P - P_h\|_Q + 
\biggl \| \frac{1}{2} \left(\uu \cdot \uu  - 
 \, \sum_{E \in \Omega_h} \Pi_k^{0, E} \uu_h \cdot  \Pi_k^{0, E} \uu_h \right) - (\lambda - \lambda_h) \biggr \|_Q
\\
&\leq 	\|P - P_h\|_Q + 
\frac{1}{2} \biggl \|  \uu \cdot \uu  - 
 \, \sum_{E \in \Omega_h} \Pi_k^{0, E} \uu_h \cdot  \Pi_k^{0, E} \uu_h \biggr \|_Q
\\
& = 	\|P - P_h\|_Q + \frac{1}{2} \, \Big( \sum_{E \in \Omega_h} \left \| \uu \cdot \uu - \Pi_k^{0, E} \uu_h \cdot  \Pi_k^{0, E} \uu_h \right\|_{Q, E}^2 \Big)^{1/2}  \\
&=:  \|P - P_h\|_Q + \frac{1}{2} \, \Big( \sum_{E \in \Omega_h} \mu_E^2 \Big)^{1/2} \,,
\end{split}
\end{equation}
where in the second inequality we have used the fact that all terms inside the norms are zero averaged.
The first term in the right hand side of \eqref{eq:bernoulliest1} is bounded by \eqref{eq:P-est}.
Whereas for the terms $\mu^E$ the triangular inequality and  the H\"older inequality entail
\begin{equation}
\label{eq:bernoulli2}
\begin{split}
\mu_E
& \leq 
 \|\uu \cdot \uu - \Pi_k^{0, E} \uu \cdot  \Pi_k^{0, E} \uu\|_E
+ \|\Pi_k^{0, E} \uu \cdot  \Pi_k^{0, E} \uu - \Pi_k^{0, E} \uu_h \cdot  \Pi_k^{0, E} \uu_h\|_E
\\
& =
 \|(\uu -  \Pi_k^{0, E} \uu) \cdot  (\uu + \Pi_k^{0, E} \uu)\|_E
+  
\|\Pi_k^{0, E} (\uu  -\uu_h) \cdot \Pi_k^{0, E} (\uu  +\uu_h)\|_E
\\
&\leq
 \|\uu -  \Pi_k^{0, E} \uu \|_{L^4(E)} \, \| \uu + \Pi_k^{0, E} \uu\|_{L^4(E)} 
+
\|\Pi_k^{0, E} (\uu  -\uu_h)  \|_{L^4(E)} \, \| \Pi_k^{0, E} (\uu  +\uu_h) \|_{L^4(E)} \,.
\end{split}
\end{equation}
Using the Sobolev embedding $H^{1}(\Omega) \subset L^4(\Omega)$,
the continuity of the projection $\Pi_k^{0, E}$ with respect
to the  $L^4$-norm and the $H^1$-norm  (see, for instance, \cite{Bdv-Lovadina-Vacca:2018}), 
and polynomial approximation estimates on star shaped polygons of Lemma \ref{fact2},
from \eqref{eq:bernoulli2} we infer
\begin{equation}
\label{eq:bernoulli3}
\begin{split}
\mu_E
&\leq
\|\uu -  \Pi_k^{0, E} \uu \|_{\VV, E} \, \| \uu + \Pi_k^{0, E} \uu\|_{\VV, E}  
+ \| \uu  -\uu_h \|_{L^4(E)} \, \| \uu  +\uu_h\|_{L^4(E)}
\\
&\leq
C \, h^s \, |\uu|_{s+1, E} \, \| \uu\|_{\VV, E} + 
 \| \uu  -\uu_h \|_{\VV, E} \, \| \uu  +\uu_h\|_{\VV, E}
\\
&\leq
C \, h^s \, |\uu|_{s+1, E} \, \| \uu\|_{\VV, E} + 
 \| \uu  -\uu_h \|_{\VV, E}   \, \left(\| \uu\|_{\VV, E}  +\|\uu_h\|_{\VV, E} \right) \,.
\end{split}
\end{equation}
Combining bound \eqref{eq:bernoulli3} with the H\"older
inequality for sequences, the velocity error estimate
\eqref{eq:thm:u} and with the stability estimates
\eqref{eq:solution estimates} and \eqref{eq:solution virtual estimates}, it follows
\begin{equation}
\label{eq:bernoulliest3}
\begin{aligned}
\Big( \sum_{E \in \Omega_h} \mu_E^2 \Big)^{1/2} &\leq 
\Big( 
\sum_{E \in \Omega_h} C \, h^{2s} \, |\uu|_{s+1, E}^2 \, \| \uu\|_{\VV, E}^2 + 
\sum_{E \in \Omega_h} \| \uu  -\uu_h \|_{\VV, E}^2   \, ( \| \uu\|_{\VV, E}^2  +\|\uu_h\|_{\VV, E}^2 ) 
\Big)^{1/2}
\\
&\leq
C \, h^s \, |\uu|_{s+1} \, \|\uu\|_{\VV} + 
 \| \uu  -\uu_h \|_{\VV}   \, \left(\| \uu\|_{\VV}  +\|\uu_h\|_{\VV} \right)
\\
&\leq
C \, h^s \, |\uu|_{s+1} \, \frac{\| \ff\|_{H^{-1}}}{\nu}
+\| \uu  -\uu_h \|_{\VV} \,
\left ( \frac{\| \ff\|_{H^{-1}}}{\nu} + 
\frac{\| \ff_h\|_{H^{-1}}}{\alpha_* \, \nu} \right)
\\
& \leq  h^s \, \mathcal{L} (\uu, \ff; \nu, \gamma, \gamma_h) + \, h^{s+2} \, \mathcal{I}(\ff; \nu, \gamma_h)  
\end{aligned}
\end{equation}
for a suitable functions $\mathcal{L}$ and $\mathcal{I}$ independent of $h$.
Finally, inserting estimates \eqref{eq:P-est} and \eqref{eq:bernoulliest3}
in \eqref{eq:bernoulliest1} we obtain the optimal convergence result
for the convective pressure also for choice \eqref{eq:crot_h^E}.
%

\begin{remark}
We observe that, due to the divergence-free property, the estimate on the velocity error in Theorem \ref{thm:u} does not depend on the continuous pressure, whereas the velocity errors of classical methods have a pressure contribution, see \cite{Bdv-Lovadina-Vacca:2018} for more details on this aspect.
\end{remark}
\begin{remark}
\label{rm:trilinear}
A careful investigation of the proof of Theorem \ref{thm:u} (see Lemma 4.3 in \cite{Bdv-Lovadina-Vacca:2018}) shows that if the exact velocity solution $\uu \in [\Pk_k(\Omega)]^2$ and the trilinear form $\cconvh(\cdot; \, \cdot, \cdot)$ or the trilinear form $\croth(\cdot; \, \cdot, \cdot)$ are adopted in  \eqref{eq:ns virtual}, the corresponding schemes provide a higher order approximation errors, that are respectively
\begin{gather*}
\|\uu - \uu_h \|_{\VV} \leq C \, h^{k+2} \, \left \| (\Gr \uu) \uu \right\|_{k+1} + C \, h^{k+2} \, \|\ff\|_{k+1} \,,
\\
\|\uu - \uu_h \|_{\VV} \leq C \, h^{k+2} \, \left \| (\cc \, \uu) \times \uu \right\|_{k+1} + C \, h^{k+2} \, \|\ff\|_{k+1} \,.
\end{gather*}
These are to be compared with the error of standard inf-sup stable
Finite Elements, that in the same situations would be $O(h^k)$
due to the pressure contribution to the velocity error.
\end{remark}
\begin{remark}\label{rem:reduced}
Another interesting aspect related to method \eqref{eq:ns virtual}
is the so called ``reduced'' version. Noting that the discrete
solution satisfies ${\rm div}\,\uu_h=0$, one can immediately
set to zero all $\mathbf{D_V4}$ degrees of freedom, and correspondingly
eliminate also the associated (locally zero average) discrete pressures.
The resulting equivalent scheme has much less internal-to-element
velocity DoFs and only piecewise constant pressures (we refer to
\cite{Stokes:divfree,Bdv-Lovadina-Vacca:2018} for more details).
\end{remark}

\section{Virtual elements Stokes complex and $\CC$ formulation}
\label{sec:5}
In the present part we present the VEM stream function space and the associated Stokes Complex.

\subsection{Virtual element space of the stream functions}
\label{sub:5.1}

The aim of the present section is to introduce a suitable virtual space $\Phi_h$ approximating the continuous space of the stream functions $\Phi$ defined in \eqref{eq:Phi}, such that 
\begin{equation}
\label{eq:Phi-Phi_h}
\CC \, \Phi_h = \ZZ_h \,.
\end{equation}
In particular this will allow to exploit the kernel formulation \eqref{eq:nsvirtual ker} in order to define an equivalent VEM approximation for the Navier--Stokes equation in  $\CC$ form (cf. \eqref{eq:ns variazionale curl}).
Note that a related approach, but limited to a lowest order case and suitable  only for the Stokes problem, was presented in \cite{Antonietti-BeiraodaVeiga-Mora-Verani:20XX}.

In order to construct the space of the virtual stream functions $\Phi_h$, we proceed step by step, following the enhanced technique used in Subsection \ref{sub:4.2}.
On each element $E \in \Omega_h$ 
we consider the enlarged local virtual space:
%
\begin{equation}
\label{eq:Psi_h}
\Psi_h^E := \biggl\{  
\phi \in H^2(E) \quad \text{s.t.} \quad
\phi_{|\partial E} \in \B_{k+1}(\partial E)\, ,\quad
(\nabla \phi)_{|\partial E} \in [\B_{k}(\partial E)]^2\, , \quad
\Delta^2  \phi \in \Pk_{k-1}(E) 
\biggr\} \,.
\end{equation}
Then we define the enhanced space of the stream functions
\begin{equation}
\label{eq:Phi_h}
\Phi_h^E := \biggl\{  
\phi \in \Psi_h^E \quad \text{s.t.} \, \, \,
\left(\CC \phi - \Pi^{\nabla,E}_k (\CC \phi), \, \mathbf{x}^{\perp} \, \widehat{p}_{k-1} \right)_{E} = 0 \quad \text{for all $\widehat{p}_{k-1} \in  \widehat{\Pk}_{k-1 \setminus k-3}(E)$}
\biggr\}\,.
\end{equation}
It is straightforward to see that $\Pk_{k+1}(E) \subseteq \Phi_h^E$.
We are now ready to introduce a suitable set of degrees of freedom for the local space of virtual stream functions $\Phi_h^E$.
Given a function $\phi \in \Phi_h^E$, 
we take the following linear operators $\mathbf{D_{\Phi}}$, split into five subsets (see Figure \ref{fig:dofslocphi})
\begin{itemize}
\item $\mathbf{D_{\Phi}1}$:  the values of $\phi$ at the vertexes of the polygon $E$,
\item $\mathbf{D_{\Phi}2}$:  the values of $\gr \phi$ at the vertexes of the polygon $E$,
\item $\mathbf{D_{\Phi}3}$: the values of $\phi$ at $k-2$ distinct points of every edge $e \in \partial E$,
\item $\mathbf{D_{\Phi}4}$: the values of $\frac{\partial \phi}{\partial n}$ at $k-1$ distinct points of every edge $e \in \partial E$,
\item $\mathbf{D_{\Phi}5}$: the moments of $\CC \, \phi$ 
\[
\int_E \CC \, \phi \cdot \mathbf{x}^{\perp}\, p_{k-3}  \, {\rm d}E \qquad \text{for all $p_{k-3} \in \Pk_{k-3}(E)$.}
\] 
\end{itemize}

\begin{figure}[!h]
\center{
\includegraphics[scale=0.25]{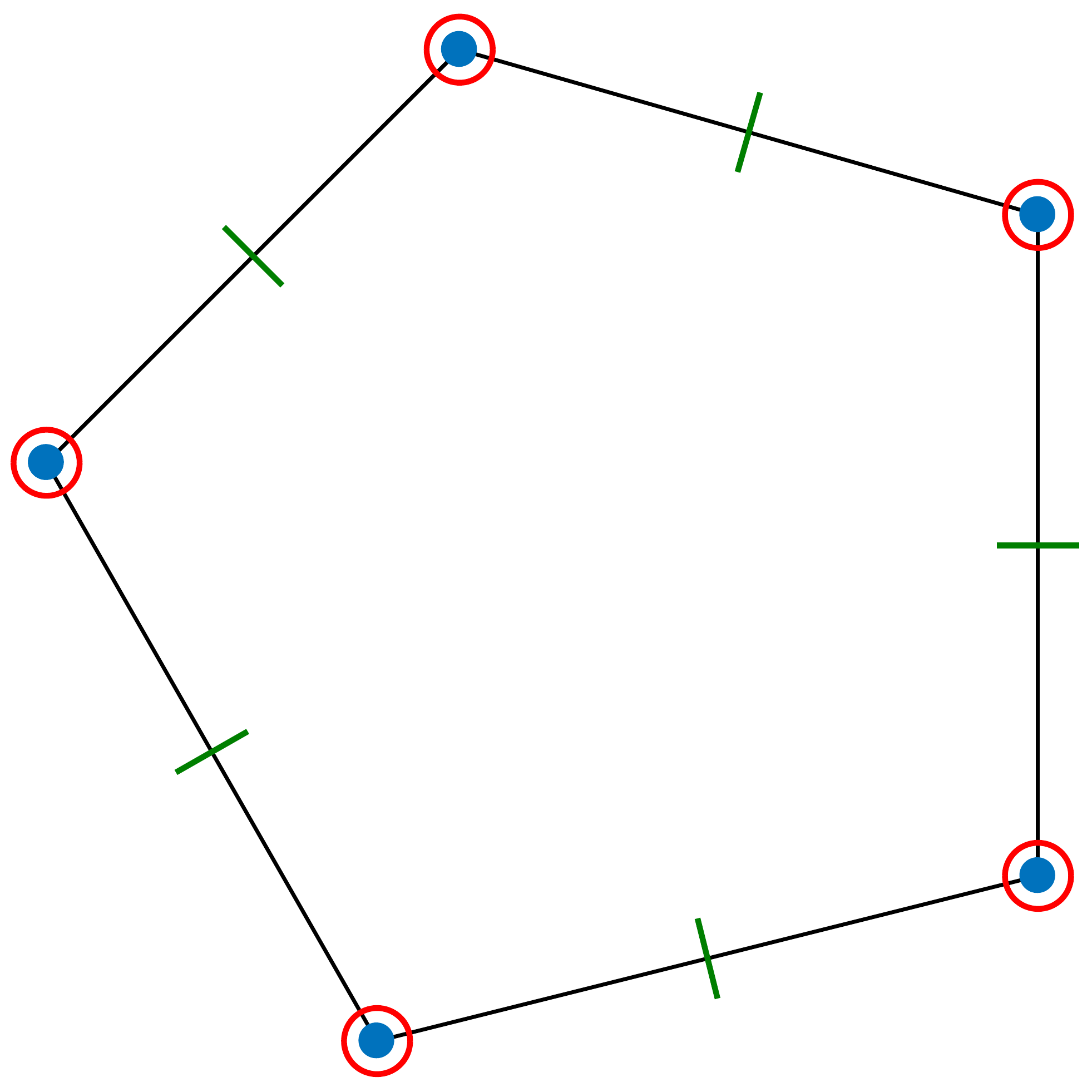} 
\qquad \qquad \qquad
\includegraphics[scale=0.25]{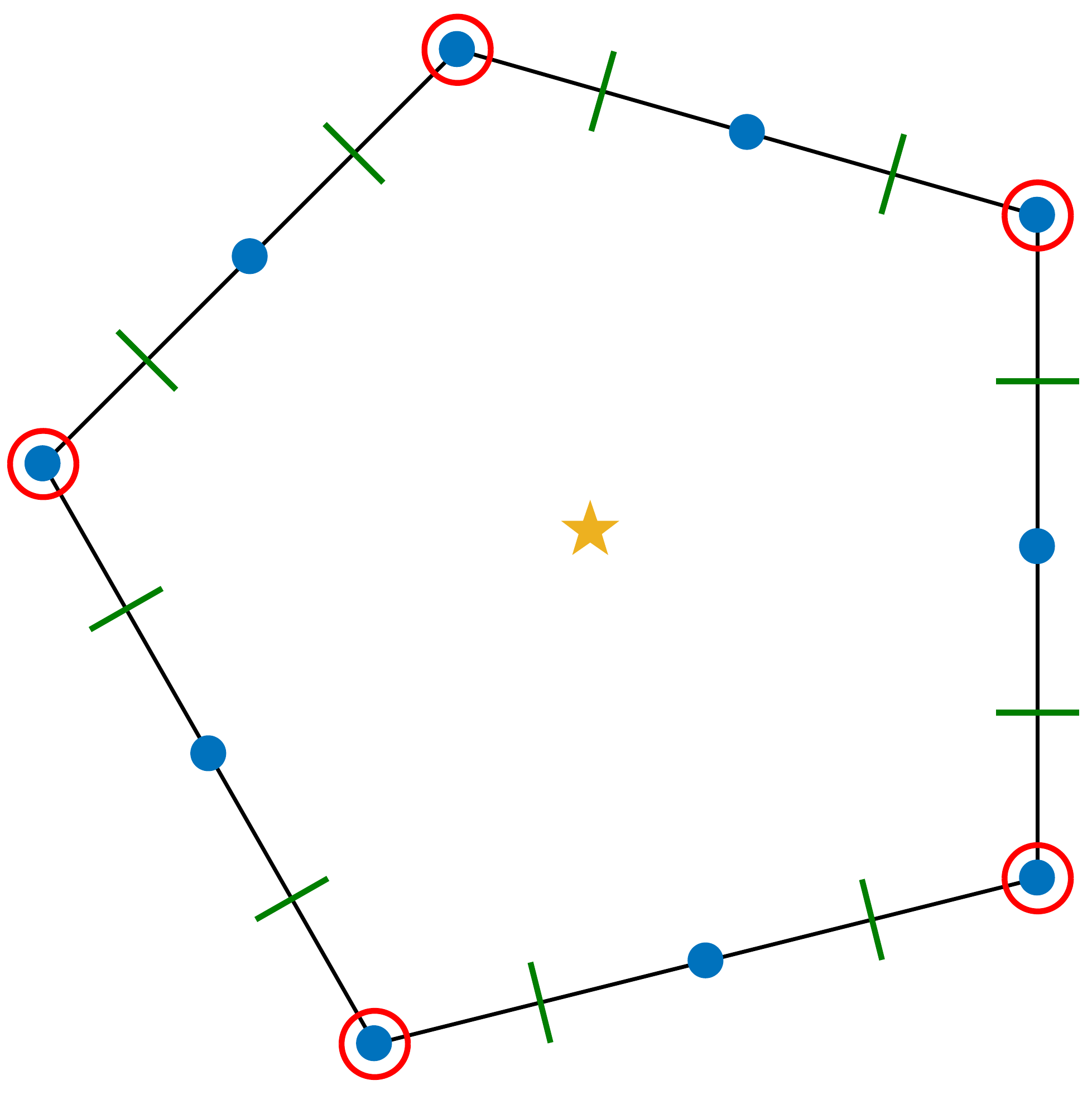}
\caption{Degrees of freedom for $k=2$ (left), $k=3$ (right). We denote 
$\mathbf{D_{\Phi}1}$ with blue dots, 
$\mathbf{D_{\Phi}2}$ with red circles, 
$\mathbf{D_{\Phi}3}$ with blued dots, 
$\mathbf{D_{\Phi}4}$ with green lines,
$\mathbf{D_{\Phi}5}$ with yellow stars.}
\label{fig:dofslocphi}
}
\end{figure}

We note that the linear operators $\mathbf{D_{\Phi}1}$ and $\mathbf{D_{\Phi}2}$ are always
needed to enforce the $C^1$-continuity at the vertices.
Moreover it is immediate to check that for any stream function $\phi \in \Phi_h^E$
(the same holds for  $\Psi_h^E$), 
the linear operator evaluations of $\mathbf{D_{\Phi}1}$,
$\mathbf{D_{\Phi}2}$, $\mathbf{D_{\Phi}3}$, $\mathbf{D_{\Phi}4}$ uniquely determine  
$\phi$ and its gradient on  $\partial E$.
We now prove the following result.
\begin{proposition}
The linear operators $\mathbf{D_{\Phi}}$ are a unisolvent set of degrees of freedom for the
virtual space of stream functions $\Phi_h^E$ and the dimension of $\Phi_h^E$ is
\begin{equation}
\label{eq:dimension_Phi}
\dim\left(\Phi_h^E \right) =  2 \, n_E \, k + \frac{(k-1)(k-2)}{2}  
\end{equation} 
where as usual $n_E$ denotes the number of edges of the polygon $E$.
\end{proposition}

\begin{proof}
We preliminary prove that
the linear operators $\mathbf{D_{\Phi}}$ plus 
the additional moments of $\CC \, \phi$ 
\[
\mathbf{D_{\Psi}5} :  \qquad 
\int_E \CC \, \phi \cdot \mathbf{x}^{\perp} \, \widehat{p}_{k-1}  \, {\rm d}E \qquad \text{for all $\widehat{p}_{k-1} \in \widehat{\Pk}_{k-1 \setminus k-3}(E)$} 
\]
constitute a set of degrees of freedom for $\Psi_h^E$. 
An integration by parts and recalling Remark \ref{fact1} imply that the linear operators $\mathbf{D_{\Psi}5} + \mathbf{D_{\Phi}5}$ are equivalent to prescribe the moments
$\int_E \phi \, q_{k-1} \, {\rm d}E$ for all $q_{k-1} \in \Pk_{k-1}(E)$.
Indeed, Remark \ref{fact1} and simple computations give
\begin{equation*}
\begin{split}
\int_E \phi \, q_{k-1}  \, {\rm d}E &= 
\int_E \phi \,  \cc ( \mathbf{x}^{\perp}\, p_{k-1}) \, {\rm d}E = 
\int_E \left(\CC \, \phi \right) \cdot \mathbf{x}^{\perp}\, p_{k-1}  \, {\rm d}E 
- \sum_{e \in \partial E} \int_{e} \phi \, \mathbf{x}^{\perp}\, p_{k-1} \cdot \mathbf{t}_e\, {\rm d}s \,,
\end{split}
\end{equation*}
where the boundary integral is computable using the DoFs values.
Now the assertion easily follows by a direct application of Proposition 4.1 in
\cite{Brezzi-Marini:2012}.
In particular, from \eqref{eq:dimensions1} it holds that
\begin{equation}
\label{eq:dimension_Psi}
\dim \left( \Psi_h^E \right) = 2 \, n_E \, k + \frac{k(k+1)}{2} \,.
\end{equation}
The next step is to prove that the linear operators $\mathbf{D_{\Phi}}$ 
are unisolvent for $\Phi_h^E$.
From \eqref{eq:dimensions1} it holds
\[
\dim \left( \widehat{\Pk}_{k-1 \setminus k-3}(E)  \right) =
\dim \left( \Pk_{k-1}(E) \right) -
\dim \left( \Pk_{k-3}(E)\right) 
= 2 \, k - 1 \,.
\]
Hence, neglecting the independence of the additional $(2\,k - 1)$
conditions in \eqref{eq:Phi_h}, 
the dimension of $\Phi_h^E$ is bounded from below by
\begin{equation}
\label{eq:dofs1}
\dim \left( \Phi_h^E \right) \geq 
\dim \left( \Psi_h^E \right) - (2\,k - 1) = 2 \, n_E \, k 
+ \frac{(k-1)(k-2)}{2} = \text{number of DoFs  $\mathbf{D_{\Phi}}$}.
\end{equation}
Due to \eqref{eq:dofs1}, the proof is concluded if we
 show that a function $\phi \in \Phi_h^E$
such that $\mathbf{D_{\Phi}}(\phi) = 0$ is identically zero.
In such case, 
$\phi = 0$ and  $\nabla \phi = \mathbf{0}$ on the skeleton $\partial E$ and this entails 
$\CC \, \phi = \mathbf{0}$ on $\partial E$. 
Moreover we note that in this case the 
$\Pi^{\nabla,E}_k (\CC \, \phi) = 0$; as a matter of fact, by definition \eqref{eq:Pn_k^E}, we get
\[
\int_E \Gr (\CC \, \phi) : \Gr \mathbf{p}_k \, {\rm d}E 
= - \int_E \CC \, \phi \cdot \dl \mathbf{p}_k \, {\rm d}E 
+ \int_{\partial E} \CC \, \phi \cdot \Gr \mathbf{p}_k 	\, \mathbf{n}_E \, {\rm d}S \,.
\]
The boundary integral is zero being $\CC \, \phi = \mathbf{0}$ on the skeleton $\partial E$. 
For the internal integral, in the light of \eqref{eq:pkgkgkp}, let us set 
$\dl \mathbf{p}_k = \gr q_{k-1} + \mathbf{x}^{\perp} \, p_{k-3}$ with
$q_{k-1} \in \Pk_{k-1}(E) \setminus \R$.
Then we infer 
\[
\begin{split}
\int_E \Gr (\CC \, \phi) : \Gr \mathbf{p}_k \, {\rm d}E 
&= 
-  \int_E \CC \, \phi  \, \cdot   \gr q_{k-1}  \, {\rm d}E 
-  \int_E \CC \, \phi \cdot   \mathbf{x}^{\perp} \, p_{k-3}  \, {\rm d}E 
\\
&=
\sum_{e \in \partial E} \int_e \phi \, \frac {\partial q_{k-1}}{\partial t} \, {\rm d}S
-  \int_E \CC \, \phi \cdot   \mathbf{x}^{\perp} \, p_{k-3}  \, {\rm d}E 
= 0 \, ,
\end{split}
\]
where the boundary integral is zero since $\phi = 0$ on $\partial E$,
whereas the second term is zero since   $\mathbf{D_{\Phi}5}(\phi)= 0$.
In particular we proved that, since $\Pi_k^{\nabla, E}( \CC \, \phi) =0$, 
recalling \eqref{eq:Phi_h} also the moments $\mathbf{D_{\Psi}5}$ of $\phi$ are zero. 
Since $\mathbf{D_{\Phi}}(\phi) = 0$ by assumption, recalling that 
$\phi \in \Phi_h^E \subset \Psi_h^E$ and that 
$\{\mathbf{D_{\Phi}}, \,  \mathbf{D_{\Psi}5}\}$ are a set of
degrees of freedom for $\Psi_h^E$, it follows $\phi = 0$.

%
\end{proof}
\begin{remark}
\label{rm:brezzi marini}
An alternative way to define a unisolvent set of DoFs for the
space $\Phi_h^E$ is to provide in the place of $\mathbf{D_{\Phi}5}$ 
the following  operators \cite{Brezzi-Marini:2012}
\begin{itemize}
\item $\widetilde{\mathbf{D_{\Phi}5}}$ : the moments of $\phi$
against the polynomial of degree up to degree $k-3$
\begin{equation}
\label{eq:moments}
\int_E \phi \, p_{k-3}  \, {\rm d}E \qquad \text{for all $p_{k-3} \in \Pk_{k-3}(E)$,}
\end{equation}
\end{itemize}
but such choice is less suitable for the exact sequence construction of the present work.

\end{remark}
The global virtual space $\Phi_h$  is obtained by combining the local spaces
$\Phi_h^E$ accordingly to the local degrees of freedom,
taking into account the boundary conditions:
\begin{equation}
\label{eq:Phi_h_global}
\Phi_h := \biggl\{  
\phi \in \Phi \quad \text{s.t.} \quad
\phi_{| E} \in \Phi_h^E  \quad \text{for all $E\in \Omega_h$}
\biggr\}\,.
\end{equation}
The dimension is
$
\dim\left(\Phi_h \right) =  3 \, n_V + (2 \, k - 3) n_e + n_P \, \frac{(k-1)(k-2)}{2}  \,,
$
where $n_P$ (resp., $n_e$ and $n_V$) is the number of elements
(resp., internal edges and vertexes) in the decomposition $\Omega_h$.


\subsection{Virtual element Stokes complex}
\label{sub:5.1bis}

The aim of the present subsection is to provide a virtual element counterpart of the 
continuous Stokes complex \cite{guzman-neilan:2014}:
\begin{equation}
\label{eq:exact_cont}
0 \, \xrightarrow[]{\, \, \, \,\quad \text{{$i$}} \quad \, \, \, \,} \,
H_0^2(\Omega) \, \xrightarrow[]{\, \quad \text{{$\CC$}} \quad \,}\,
[H_0^1(\Omega)]^2 \, \xrightarrow[]{\, \quad \text{{$\dd$}} \quad \,} 
L_0^2(\Omega) \, \xrightarrow[]{\, \, \quad \text{{$0$}} \quad \, \,} 
0 \,,
\end{equation}
where $i$  denotes the mapping that to every real number $r$
associates the constant function identically equal to $r$
and we recall that a sequence is exact if the image of
each operator coincides with the kernel of the following one. 
The case without boundary conditions is handled analogously.

We start by characterizing the space $\Phi_h^E$ as the
space of the stream functions associated to the discrete kernel $\ZZ_h^E$.
\begin{proposition}
\label{prp:stream}
For any $E \in \Omega_h$ let $\ZZ_h^E$ and $\Phi_h^E$ be the
spaces defined in \eqref{eq:Z_h^E} and \eqref{eq:Phi_h}, respectively.
Then, it holds that
\[
\CC \, \Phi_h^E = \ZZ_h^E \,.
\]
\end{proposition}
\begin{proof}
For any $\phi_h \in \Phi_h^E$, we show that the function $\vv_h : = \CC \, \phi_h \in \ZZ_h^E$.
First of all, it is straightforward to check that
\begin{equation}
\label{eq:cond0}
\dd \, \vv_h = \dd \left(\CC \, \phi_h \right) = 0 \,.
\end{equation}
Concerning the condition on the  skeleton $\partial E$, we observe that 
${\phi_h}_{| \partial E} \in \B_{k+1}(\partial E)$ and $(\gr \phi_h)_{| \partial E} \in [\B_{k}(\partial E)]^2$ easily imply that 
\begin{equation}
\label{eq:cond1}
{\vv_h}_{| \partial E}  = (\CC \,\phi_h)_{| \partial E} \in [\B_{k}(\partial E)]^2 \,.
\end{equation}
Inside the element, by simple calculations and by definition \eqref{eq:Phi_h}, we infer
$
\cc \, \dl \vv_h  =
\cc \, \dl (\CC \, \phi_h) =  
\Delta ^2\phi_h
\in \Pk_{k-1}(E)
$.
In the light of Remark \ref{fact1}, the previous relation is equivalent to
$$
\cc \, \dl \vv_h  
\, \in \, \cc  \left( \mathbf{x}^{\perp} \, \Pk_{k-1}(E) \right) \,.
$$
Therefore, there exists $p_{k-1} \in \Pk_{k-1}(E)$
such that $\cc ( \dl \vv_h - \mathbf{x}^{\perp} \, p_{k-1}) = 0$. 
Since $E$ is simply connected, there exists $s$ such that
$\dl \vv_h - \mathbf{x}^{\perp} \, p_{k-1} = \nabla s$.
Thus we have shown that
\begin{equation}
\label{eq:cond2}
(\dl \vv_h -  \nabla \, s) \, \in \, \mathbf{x}^{\perp} \, \Pk_{k-1}(E)\,.
\end{equation}
Moreover, by definition \eqref{eq:Phi_h}, for all
$\widehat{p}_{k-1} \in  \widehat{\Pk}_{k-1 \setminus k-3}(E)$ it holds
\begin{equation}
\label{eq:cond3}
\left(\vv_h - \Pi^{\nabla,E}_k \, \vv_h, \, \mathbf{x}^{\perp} \, \widehat{p}_{k-1} \right)_{E} = 
\left(\CC \phi_h - \Pi^{\nabla,E}_k (\CC \phi_h), \, \mathbf{x}^{\perp} \, \widehat{p}_{k-1} \right)_{E} = 0 \,.
\end{equation}
At this point is clear that \eqref{eq:cond0}, \eqref{eq:cond1}, \eqref{eq:cond2},
\eqref{eq:cond3} and definition \eqref{eq:Z_h^E}, imply $\vv_h = \CC \, \phi_h \in \ZZ_h^E$ for any scalar potential $\phi_h \in \Phi_h^E$,
i.e. 
$\CC \, \Phi_h^E \subseteq \ZZ_h^E$.
The proof now follows by a dimensional argument.
In fact, from \eqref{eq:dimension_Phi} and \eqref{eq:dimensione Z_h^E} easily follows that 
\[
\dim \left(\CC(\Phi_h) \right) =
 \dim \left(\Phi_h^E \right) - 1 =
2 \,  n_E \, k + \frac{(k-2)(k-1)}{2}   - 1 = \dim \left( \ZZ_h^E \right) \,.
\]
Therefore we can conclude that $\CC \, \Phi_h^E = \ZZ_h^E$ for all $E \in \Omega_h$.
\end{proof}
\begin{remark}
\label{rm:dofs}
Given any $\phi_h \in \Phi_h^E$, from the degrees of freedom values $\mathbf{D_{\Phi}}$ of $\phi_h$,
we are able to compute the DoFs values $\mathbf{D_V}$ of $\CC \, \phi_h$.
In particular it holds that 
\begin{gather*}
\mathbf{D_V3}(\CC \, \phi_h) = \mathbf{D_{\Phi}5}(\phi_h)
\qquad
\text{and}
\qquad
\mathbf{D_V4}(\CC \, \phi_h) = 0 \,.
\end{gather*}
Therefore, for any $\phi_h \in \Phi_h^E$, the DoFs $\mathbf{D_{\Phi}}$ allow to compute
the polynomial projections
$\PN (\CC \, \phi_h)$, $\P0 (\CC \, \phi_h)$ and $\PP0\nabla(\CC \, \phi_h)$.
\end{remark}

As a consequence of Proposition~\ref{prp:stream} we have the following
Stokes exact sequence for our discrete VEM spaces and its reduced
version (see also Figures~\ref{fig:stokes_c} and \ref{fig:stokes_c_r}).
\begin{corollario}
\label{cor:exact}
Let  $\Phi_h^E$ and $\VV_h^E$  the spaces defined in  \eqref{eq:Phi_h}
and \eqref{eq:V_h^E}, respectively,
and let $\widetilde{\VV}_h^E$ denote the reduced velocity space, see Remark~\ref{rem:reduced}.
Then, the following sequences are exact 
\begin{equation}
\label{eq:exact3}
\R \, \xrightarrow[]{\, \, \, \,\quad \text{{$i$}} \quad \, \, \, \,} \,
\Phi_h^E \, \xrightarrow[]{\, \quad \text{{$\CC$}} \quad \,}\,
{\VV_h^E} \, \xrightarrow[]{\, \quad \text{{$\dd$}} \quad \,} 
{\Pk_{k-1}(E)} \, \xrightarrow[]{\, \, \quad \text{{$0$}} \quad \, \,} 
0 \,,
\end{equation}
\begin{equation}
\label{eq:exact3bis}
\R \, \xrightarrow[]{\, \, \, \,\quad \text{{$i$}} \quad \, \, \, \,} \,
\Phi_h^E \, \xrightarrow[]{\, \quad \text{{$\CC$}} \quad \,}\,
{\widetilde{\VV}_h^E} \, \xrightarrow[]{\, \quad \text{{$\dd$}} \quad \,} 
\, {\Pk_{0}(E)} \, \, \xrightarrow[]{\, \, \quad \text{{$0$}} \quad \, \,} 
0 \,.
\end{equation}
\end{corollario}

\vspace{0.5cm}

\begin{figure}[!h]
\center
{
\begin{overpic}[scale=0.57]{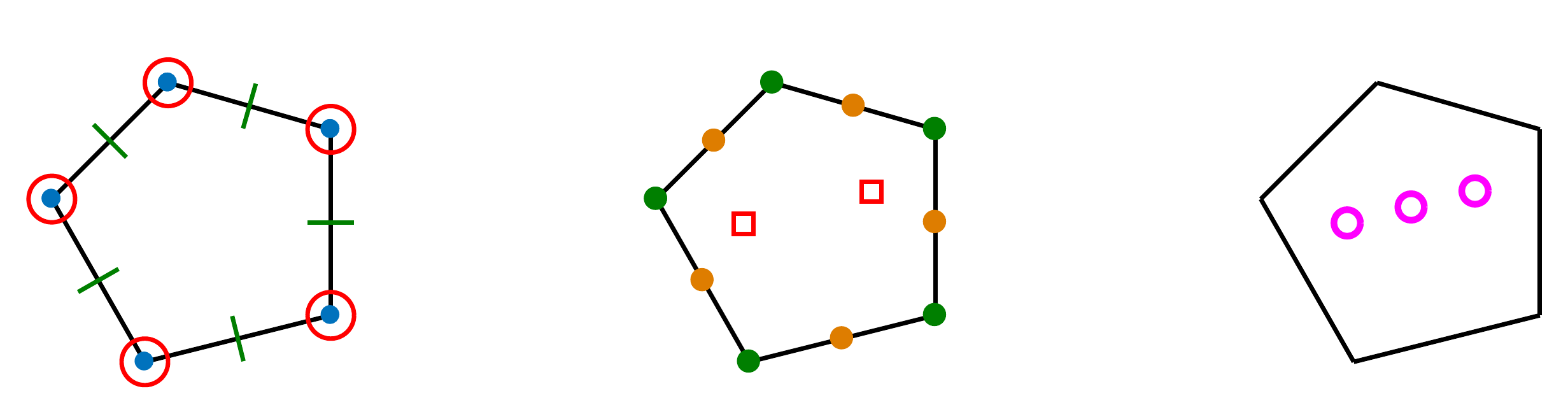} 
\put (11,23) {$\Phi_h^E$}
\put (50,23) {$\VV_h^{E}$}
\put (87,23) {$\Pk_{k-1}(E)$}
\put (25,12) {$\xrightarrow[]{\, \quad \text{{$\CC$}} \quad \,}$}
\put (65,12) {$\xrightarrow[]{\, \quad \text{{$\dd$}} \quad \,}$}
\end{overpic}
\caption{$H^2$-conforming stream virtual element $\Phi_h^E$ (left), $H^1$-conforming  virtual velocity space $\VV_h^E$ (middle), and  pressure
space $\Pk_{k-1}(E)$ (right) satisfying the exact complex \eqref{eq:exact3}.}
\label{fig:stokes_c}
}
\end{figure}

\vspace{0.5cm}

\begin{figure}[!h]
\center
{
\begin{overpic}[scale=0.57]{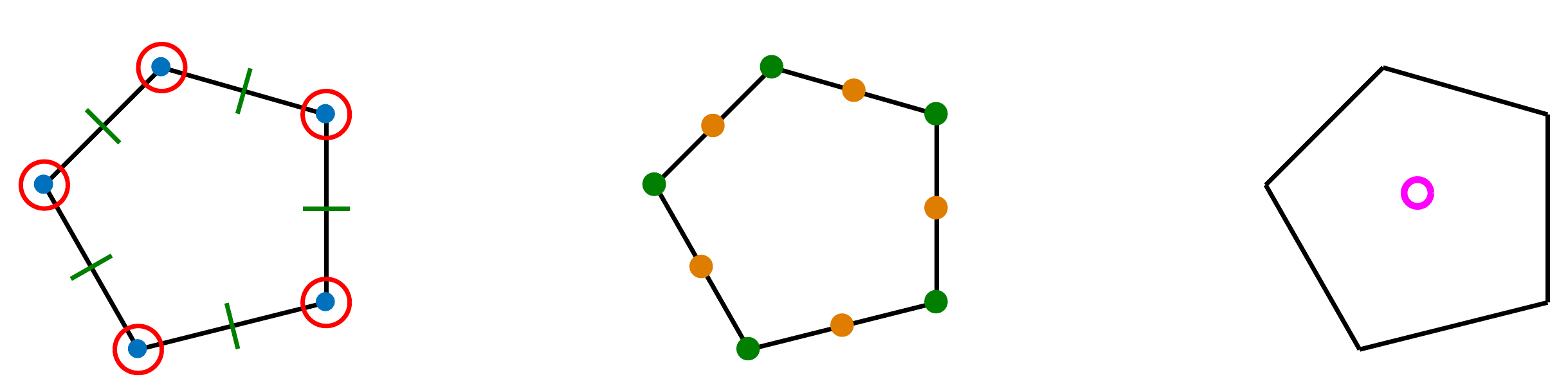} 
\put (11,23) {$\Phi_h^E$}
\put (50,23) {$\widetilde{\VV}_h^{E}$}
\put (87,23) {$\Pk_{0}(E)$}
\put (25,12) {$\xrightarrow[]{\, \quad \text{{$\CC$}} \quad \,}$}
\put (65.5,12) {$\xrightarrow[]{\, \quad \text{{$\dd$}} \quad \,}$}
\end{overpic}
\caption{Reduced version of the Virtual Element Stokes complex
(see Remark \ref{rem:reduced}).}
\label{fig:stokes_c_r}
}
\end{figure}
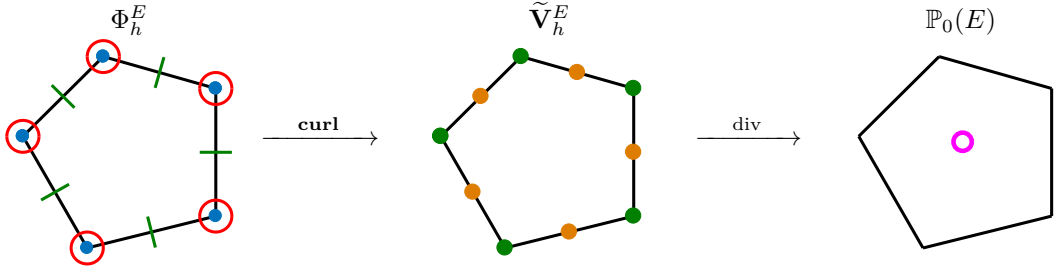

\begin{remark}
In terms of degrees of freedom, our lowest order element
(when restricted to triangles!) can be compared with the
Zienkiewicz element \cite{ciarlet:1978, guzman-neilan:2014},
all other FEM elements in the literature being either higher
order or needing a sub-element partition and more DoFs.
The reduced version of our VEM element for $k=2$
(see Remark \ref{rem:reduced} and Figure \ref{fig:stokes_c_r}) has piecewise constant
pressures and no internal degrees of freedom for velocities,
and thus in terms of degrees of freedom exactly corresponds
to the above finite element (the difference is that we use
the VE approach instead of introducing rational basis functions).
But note that the element here presented yields $O(h^2)$
convergence rate for velocities and also for the local
pressure average (full $O(h^2)$ pressure convergence can be
recovered by a local post-processing), instead of linear
convergence as \cite{guzman-neilan:2014}. 
In addition, we avoid integration of rational functions.
Clearly, this comes at the price of having a virtual formulation
and thus the absence of an explicit expression of the shape functions. 
\end{remark}

The following results are the global counterpart of
Proposition~\ref{prp:stream} and Corollary \ref{cor:exact}.
\begin{proposition}
\label{prp:stream global}
Let $\ZZ_h$ and $\Phi_h$ be the spaces defined in \eqref{eq:Zh} and
\eqref{eq:Phi_h_global}, respectively.
Then, it holds that
\[
\CC \, \Phi_h = \ZZ_h \,.
\]
\end{proposition}
\begin{proof}
We note that Proposition \ref{prp:stream} endowed with the boundary
condition in the definitions  \eqref{eq:Zh} and \eqref{eq:Phi_h_global}
imply $\CC \, \Phi_h \subseteq \ZZ_h$.
The proof now follows by a dimensional argument using the Euler formula.
\end{proof}

\begin{corollario}
\label{cor:exact global}
Let $\Phi_h^E$, $\VV_h$ and $Q_h$ be the spaces defined in
\eqref{eq:Phi_h_global}, \eqref{eq:V_h} and \eqref{eq:Q_h}, respectively,
and let $\widetilde{\VV}_h$ and $\widetilde{Q}_h$ denote the reduced
velocity space and the piecewise constant pressures, respectively,
see Remark \ref{rem:reduced}.
Then, the following sequences are exact 
\begin{equation}
\label{eq:exact4}
0 \, \xrightarrow[]{\, \, \, \,\quad \text{{$i$}} \quad \, \, \, \,} \,
\Phi_h \, \xrightarrow[]{\, \quad \text{{$\CC$}} \quad \,}\,
{\VV_h} \, \xrightarrow[]{\, \quad \text{{$\dd$}} \quad \,} 
{Q_h} \, \xrightarrow[]{\, \, \quad \text{{$0$}} \quad \, \,} 
0 \,,
\end{equation}
\begin{equation}
\label{eq:exact4bis}
0 \, \xrightarrow[]{\, \, \, \,\quad \text{{$i$}} \quad \, \, \, \,} \,
\Phi_h \, \xrightarrow[]{\, \quad \text{{$\CC$}} \quad \,}\,
{\widetilde{\VV}_h} \, \xrightarrow[]{\, \quad \text{{$\dd$}} \quad \,} 
\, {\widetilde{Q}_h} \, \, \xrightarrow[]{\, \, \quad \text{{$0$}} \quad \, \,} 
0 \,.
\end{equation}
The case without boundary conditions follows analogously.
\end{corollario}

\subsection{The discrete problem}
\label{sub:5.2}
In the light of Proposition \ref{prp:stream}, referring to
\eqref{eq:nsvirtual ker} and \eqref{eq:Phi_h_global}
we can set the virtual element approximation of the
Navier--Stokes equation in the $\CC$ formulation:
\begin{equation}
\label{eq:vem-curl}
\left\{
\begin{aligned}
& \text{find $\psi_h \in \Phi_h$, such that} \\
& \nu \, a_h(\CC \, \psi_h, \, \CC \, \phi_h) + c_h(\CC \, \psi_h; \, 
\CC \, \psi_h, \, \CC \, \phi_h) = (\ff_h, \, \CC \, \phi_h) \, & \text{for all $\phi_h \in \Phi_h$.}
\end{aligned}
\right.
\end{equation}
Due to Proposition \ref{prp:stream global}, Problem~\eqref{eq:vem-curl}
is equivalent to \eqref{eq:nsvirtual ker}.
We remark that all forms in \eqref{eq:vem-curl} are exactly
computable by the DoFs $\mathbf{D_{\Phi}}$. In fact, recalling
Remark~\ref{rm:dofs}, the polynomials
$\PN (\CC \, \phi_h)$, $\P0 (\CC \, \phi_h)$ and $\PP0\nabla(\CC \, \phi_h)$
are computable on the basis of $\mathbf{D_{\Phi}}$, so that,
referring to \eqref{eq:a_h}, \eqref{eq:c_h} and \eqref{eq:right}, we infer that
\[
a_h(\CC \, \cdot, \, \CC \, \cdot) \,, \qquad 
c_h(\CC \, \cdot; \, \CC \, \cdot, \, \CC \, \cdot) \,, \qquad 
(\ff_h, \, \CC \, \cdot)
\]
are exactly computable from DoFs $\mathbf{D_{\Phi}}$, see also Remark \ref{curl-operator}.
The well-posedness of Problem \eqref{eq:vem-curl} under
the assumption $\mathbf{(A0)_h}$  
immediately follows from Corollary \ref{cor:exact global}. 
In fact  the $\CC$
operator on the global space $\Phi_h$ is an isomorphism into $\ZZ_h$. 
Owing to this isomorphism, the existence and uniqueness result of Theorem \ref{thm:well} carry over to Problem \eqref{eq:vem-curl}, yielding the following theorem.
\begin{theorem}
\label{thm:wellstream}
Under
the assumption $\mathbf{(A0)_h}$  
Problem \eqref{eq:vem-curl} has a unique solution $\psi_h \in \Phi_h$ such that
\begin{equation}
\label{eq:solution virtual estimates curl}
\| \psi_h\|_{\Phi} \leq \frac{\| \ff_h\|_{H^{-1}}}{\alpha_* \, \nu} \,.
\end{equation}
\end{theorem}
The convergence
of the discrete solution $\CC \, \psi_h$ of \eqref{eq:vem-curl} 
to the continuous solution $\CC \, \phi$ of \eqref{eq:ns variazionale curl}
follows immediately from Theorem \ref{thm:u}, taking $\uu = \CC \, \phi$ and
$\uu_h = \CC \, \phi_h$.

%
%

Clearly Problem \eqref{eq:vem-curl} does not provide any information on the pressure $p$.
Nevertheless, the Stokes complex associated to the proposed scheme turns out to be very helpful 
if we are interested in computing suitable approximation $p_h$ of $p$.
Indeed referring to \eqref{eq:V_h} and \eqref{eq:ns virtual},  starting from the solution $\psi_h$ of Problem \eqref{eq:vem-curl}, we infer the following problem
\begin{equation}
\label{eq:pressure}
\left \{
\begin{aligned}
& \text{find $p_h \in Q_h$, such that}
\\
& b(\vv_h, \, p_h) = 
- \nu \, a_h(\CC \, \psi_h, \, \vv_h) 
- c_h(\CC \, \psi_h; \,  \CC \, \psi_h, \, \vv_h) +  (\ff_h, \, \vv_h) 
\qquad  \text{for all $\vv_h \in \VV_h$.}
\end{aligned}
\right .
\end{equation}
Since $\dim(\VV_h) > \dim(Q_h)$, the previous system, is actually an  overdetermined system, i.e. there are more equations than unknowns.
Nevertheless the well-posedness  of Problem \eqref{eq:pressure} is guaranteed by Theorem \ref{thm:well}.
We refer to Section \ref{sec:tests} for a deeper analysis and computational aspects of \eqref{eq:pressure}.

We stress that the $\CC$ virtual formulation \eqref{eq:vem-curl} exhibits important differences
from the computational point of view compared with the velocity-pressure formulation \eqref{eq:ns virtual}.
First of all the linear system associated to Problem \eqref{eq:vem-curl} has
$2(n_P - 1)$ less DoFs than Problem  \eqref{eq:ns virtual}, even if considering its equivalent reduced form (see Remark \ref{rem:reduced}).
Moreover the first iteration of the Newton method applied to the 
the non-linear virtual stream formulation
\eqref{eq:vem-curl} results in a linear system which is symmetric and positive definite,
whereas applied to the virtual element method \eqref{eq:ns virtual} in velocity-pressure formulation leads
to an indefinite linear system. 
These advantages come at the price of a higher condition number of the involved linear systems.

\begin{remark}
\label{rm:load}
Simple integration by parts gives
$
(\ff, \, \CC \, \phi) = (\cc \, \ff, \, \phi)
$.
By Remark \ref{rm:brezzi marini}, the DoFs $\mathbf{D_{\Phi}}$ allow us to compute  the $L^2$-projection $\Pi^{0,E}_{k-1} \colon \Phi_h^E \to \Pk_{k-1}(E)$, so that we can consider a  new computable right hand-side
\begin{equation}
\label{eq:right new}
\left((\cc \, \ff)_h, \phi_h \right)  := 
\sum_{E \in \Omega_h} \int_E \Pi_{k-1}^{0, E} (\cc \, \ff) \, \phi_h \, {\rm d}E = 
\sum_{E \in \Omega_h} \int_E (\cc \, \ff) \, \Pi_{k-1}^{0,E}  \phi_h \, {\rm d}E \,.
\end{equation}
This new formulation of the right-hand side gets the same order of accuracy of the original one.
 
In particular if the external force is irrotational, i.e. $\ff = \nabla f$, we improve the error estimate in \eqref{eq:thm:u} by removing the dependence of the error by the load.
More generally, with the choice \eqref{eq:right new}, 
we completely remove the influence in the error stemming  from the irrotational part in the Helmholtz decomposition of the load.
Clearly \eqref{eq:right new} can be applied only when $f$ is given as an explicit function.
\end{remark}

\section{Numerical Tests}
\label{sec:tests}

In this section we present two sets of numerical experiments to test the practical performance of the proposed virtual element methods \eqref{eq:vem-curl}, also compared with a direct $C^1$ VEM discretization of the stream formulation \eqref{eq:ns stream} described in the Appendix, see equation \eqref{eq:ns stream vem}. 
For the scheme \eqref{eq:vem-curl}, in all tests we investigate the three possible options for the trilinear form  in \eqref{eq:cconv_h^E}, \eqref{eq:cskew_h^E}, \eqref{eq:crot_h^E}.
In Test \ref{test1}  we study the convergence of the proposed
virtual element schemes \eqref{eq:vem-curl} and \eqref{eq:ns stream vem} for the  discretization of the Navier--Stokes equation
in $\CC$ formulation and stream formulation respectively.
A comparison of \eqref{eq:vem-curl} (in terms of errors, number of DoFs, condition number of the resulting linear systems) with the equivalent virtual element scheme \eqref{eq:ns virtual} for the Navier--Stokes equation in  velocity-pressure formulation is also performed. 
In Test \ref{test2} we consider a benchmark problem for the Navier--Stokes equation \eqref{eq:ns stream} with the property of having the velocity and stream solution in the corresponding  discrete spaces. It is well known that classical mixed finite element methods lead to significant velocity errors, stemming from the velocity/pressure coupling in the error estimates. This effect is greatly reduced by the presented methods (cf. Theorem \ref{thm:u}, estimate \eqref{eq:thm:u} and Remark \ref{rm:trilinear}). 
In order to compute the VEM errors, we consider the computable error quantities:
\begin{gather*}
\text{\texttt{error}}(\mathbf{u}, H^1) := \left( \sum_{E \in \Omega_h} \left \| \boldsymbol{\nabla} \, \uu -  \boldsymbol{\Pi}_{k-1}^{0, E} (\boldsymbol{\nabla} \, \uu_h) \right \|_{0,E}^2 \right)^{1/2} 
\end{gather*}
for the velocity-pressure formulation \eqref{eq:ns virtual} and 
\[
\text{\texttt{error}}(\psi, H^2) := \left( \sum_{E \in \Omega_h} \left \| \boldsymbol{\nabla} \, \CC \psi -  \boldsymbol{\Pi}_{k-1}^{0, E} (\boldsymbol{\nabla} \,  \CC \psi_h) \right \|_{0,E}^2 \right)^{1/2} 
\]
for the $\CC$ and stream formulations (see \eqref{eq:vem-curl}
and \eqref{eq:ns stream vem}, respectively).
For what concerns the pressures we simply compute the standard $L^2$ error
$
\text{\texttt{error}}(p, L^2)  := \|p - p_h\|_0 
$.
For the computation of the discrete pressure for the virtual element scheme \eqref{eq:vem-curl} 
we follow \eqref{eq:pressure} and solve the overdetermined system by means of the least squares method.
We briefly sketch the construction of the least square formula.
Let $\{\vv_j\}_{j=1}^{{\rm dim}(\VV_h)}$ be  the canonical basis functions of $\VV_h$
and let us denote with  $\overline{\boldsymbol{r}}_h$ the vector with component
\[
\overline{\boldsymbol{r}}_{h, j} :=
- \nu \, a_h(\CC \, \psi_h, \, \vv_j) 
- c_h(\CC \, \psi_h; \,  \CC \, \psi_h, \, \vv_j) +  (\ff_h, \, \vv_j) 
\qquad  \text{for $j=1, \dots, {\rm dim}(\VV_h)$.}
\]
i.e. $\overline{\boldsymbol{r}}_h$ contains the values of the degrees of freedom associated
to the right hand side of \eqref{eq:pressure} with respect to the basis $\{\vv_j\}_{j=1}^{{\rm dim}(\VV_h)}$.
Similarly, let $\{q_i\}_{i=1}^{{\rm dim}(Q_h)}$ be  the canonical basis functions of $Q_h$ and for any piecewise polynomial $p_{k-1} \in Q_h$ we denote with 
$\overline{\boldsymbol{p}}_h$ the vector containing the values of the
coefficients with respect to the basis $\{q_i\}$ associated to $p_{k-1}$.
Then the least squares formula associated to \eqref{eq:pressure} is
\[
B \,B^T \, \overline{\boldsymbol{p}}_h = B \, \overline{\boldsymbol{r}}_h
\]  
where the matrix $B \in \R^{{\rm dim}(Q_h) \times {\rm dim}(\VV_h)}$ is defined by
\[
B_{i,j} = b(\vv_j, \, q_i) 
\qquad
\text{for $i=1, \dots, {\rm dim}(Q_h)$ and $j=1, \dots, {\rm dim}(\VV_h)$.}
\]
\begin{remark}\label{curl-operator}
We note that, if a code for formulation \eqref{eq:ns virtual} is available, all is needed in order to implement \eqref{eq:vem-curl} is the construction of the rectangular matrixes that represent, in terms of the degrees of freedom, the local operator $\CC: \Phi_h^E \rightarrow \VV_h^E$. Once this matrixes are built, one just needs to combine them with the local stiffness matrixes of scheme \eqref{eq:vem-curl} and finally assemble the global systems for the space $\Phi_h$. 
\end{remark} 
The polynomial degree of accuracy for the numerical tests is $k=2$.
In the experiments we consider the computational domains $\Omega_{\rm Q} := [0,1]^2$ and $\Omega_{\rm D} := \{\mathbf{x} \in \R^2 \, \text{s.t.} \, |\mathbf{x}| \leq 1 \}$. The square domain $\Omega_{\rm Q}$ is partitioned using the following sequences of polygonal meshes:
\begin{itemize}
\item $\{ \mathcal{V}_h\}_h$: sequence of CVT (Centroidal Voronoi tessellation) with $h=1/8, 1/16, 1/32, 1/64$,
\item $\{ \mathcal{Q}_{h} \}_h$: sequence of distorted quadrilateral meshes with $h=1/10, 1/20, 1/40, 1/80$.
\end{itemize}
An example of the adopted meshes is shown in Figure \ref{meshq}.  
For the generation of the Voronoi meshes we use the code Polymesher \cite{TPPM12}.
The distorted quadrilateral meshes are obtained starting from the uniform square meshes and displacing the internal vertexes with a proportional ``distortion amplitude'' of $0.3$. 
\begin{figure}[!h]
\centering
\includegraphics[scale=0.27]{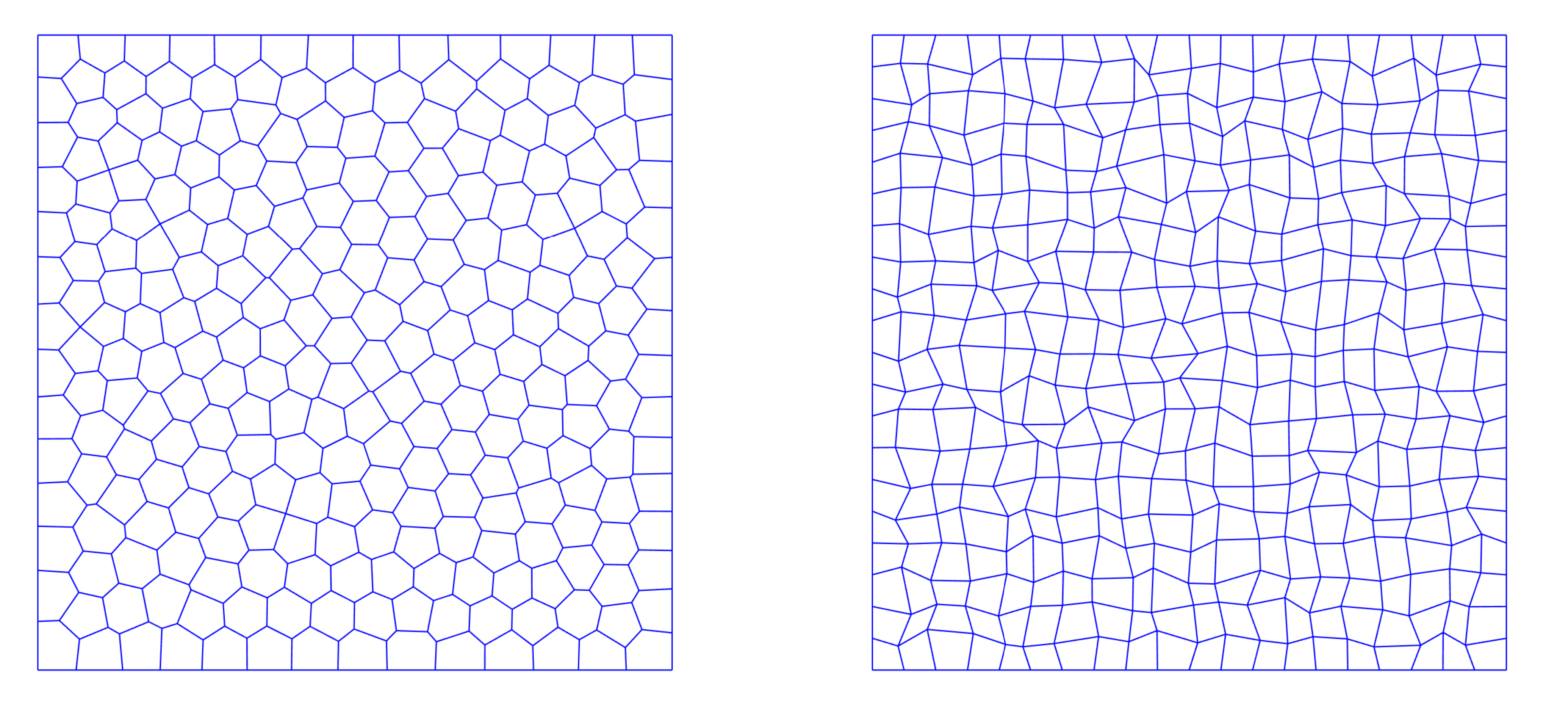} 
\caption{Example of polygonal meshes: $\mathcal{V}_{1/16}$,  $\mathcal{Q}_{1/20}$.}
\label{meshq}
\end{figure}
For what concerns the disk $\Omega_{\rm D}$ we consider the sequences of polygonal meshes:
\begin{itemize}
\item $\{ \mathcal{T}_h\}_h$: sequence of triangular meshes with $h=1/5, 1/10, 1/20, 1/40$,
\item $\{ \mathcal{W}_h\}_h$: sequence of mapped CVT with $h= 1/4, 1/8, 1/16, 1/20$.
\end{itemize}
The meshes $\mathcal{W}_h$ are obtained by mapping a CVT on the square 
$[-1, 1]^2$ on the disk $\Omega_{\rm D}$ through the map
\[
\Sigma \colon [-1, 1]^2 \to  \Omega_{\rm D} \,
\qquad \text{with} \qquad
\Sigma \colon (x, y) \mapsto 
\left( x \sqrt{1 - \frac{y^2}{2}}, \, y \sqrt{1 - \frac{x^2}{2}}\right) \,.
\] 
Figure \ref{meshd} displays an example of the adopted meshes.
\begin{figure}[!h]
\centering
\hspace{-0.5cm}
\includegraphics[scale=0.3]{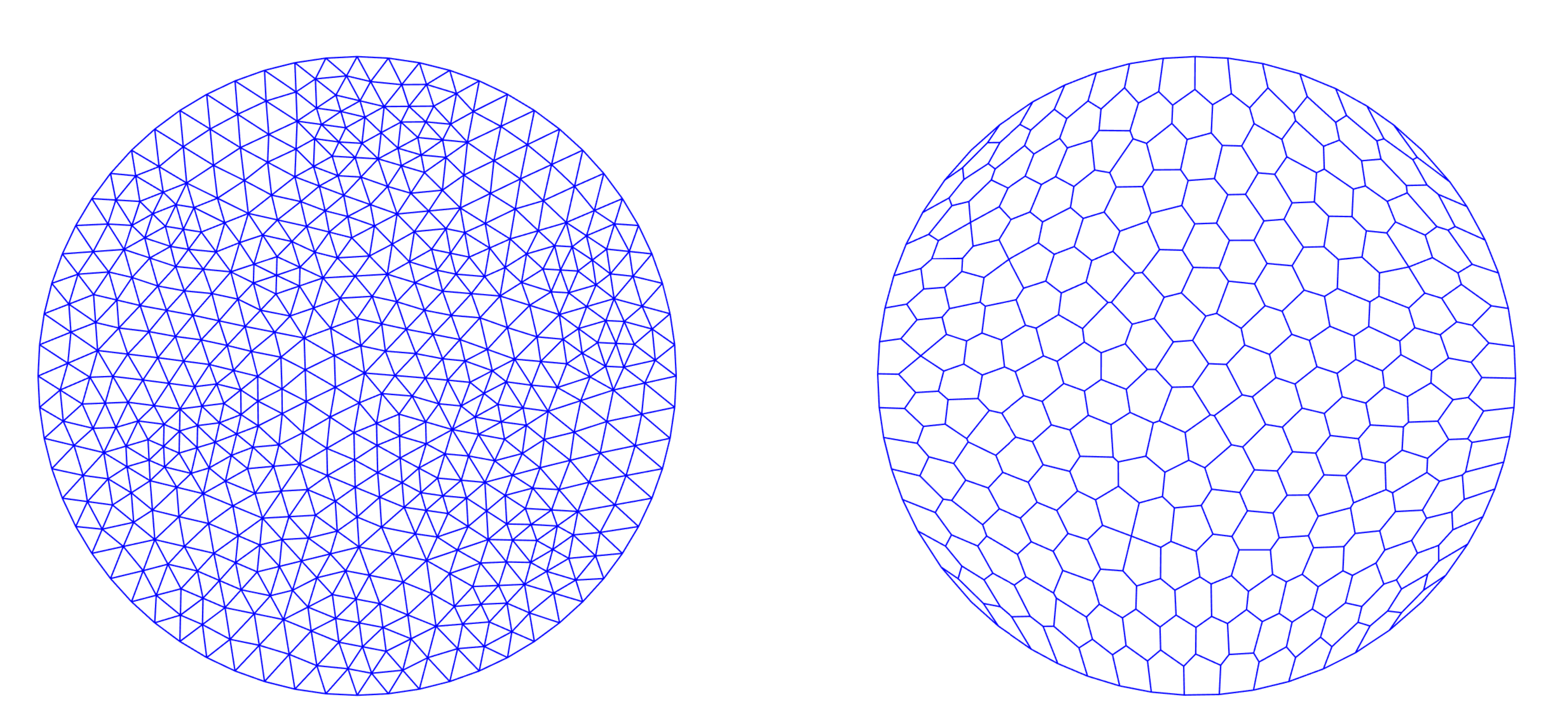} 
\caption{Example of polygonal meshes: $\mathcal{T}_{1/10}$,  $\mathcal{W}_{1/8}$.}
\label{meshd}
\end{figure}

\begin{test}
\label{test1}
In this test we solve the Navier--Stokes equation 
on the square domain $\Omega_{\rm Q}$ with  viscosity $\nu = 1$ and with
the load term $\mathbf{f}$ chosen such that the analytical velocity-pressure solution 
and the corresponding stream function solution are respectively 
\begin{gather*}
\mathbf{u}(x,y) = \frac{1}{2} \, \begin{pmatrix}
\sin(2 \pi x)^2 \, \sin(2 \pi y) \, \cos(2 \pi y) \\
- \sin(2 \pi y)^2 \, \sin(2 \pi x) \, \cos(2 \pi x)
\end{pmatrix} \qquad 
p(x,y) = \pi^2 \, \sin(2 \pi x) \, \cos(2 \pi y) \,,
\\
\psi(x,y) = 
\frac{1}{8\, \pi} \, \sin(2 \pi x)^2 \, \sin(2 \pi y)^2 \,.
\end{gather*}

We test the virtual element scheme \eqref{eq:vem-curl}. 
In Figures \ref{fig:h-voronoi}
and \ref{fig:h-quad} we show the results obtained with  the sequences of Voronoi meshes $\mathcal{V}_h$ and  quadrilateral meshes $\mathcal{Q}_h$, by considering the three possible choices of the trilinear forms.
We stress that in all cases considered we compare the discrete convective pressure $p_h$ (for the trilinear form in $\croth(\cdot; \cdot, \cdot)$ we consider the definition \eqref{eq:ph_convective}).  
\begin{figure}[!h]
\centering
\includegraphics[scale=0.3]{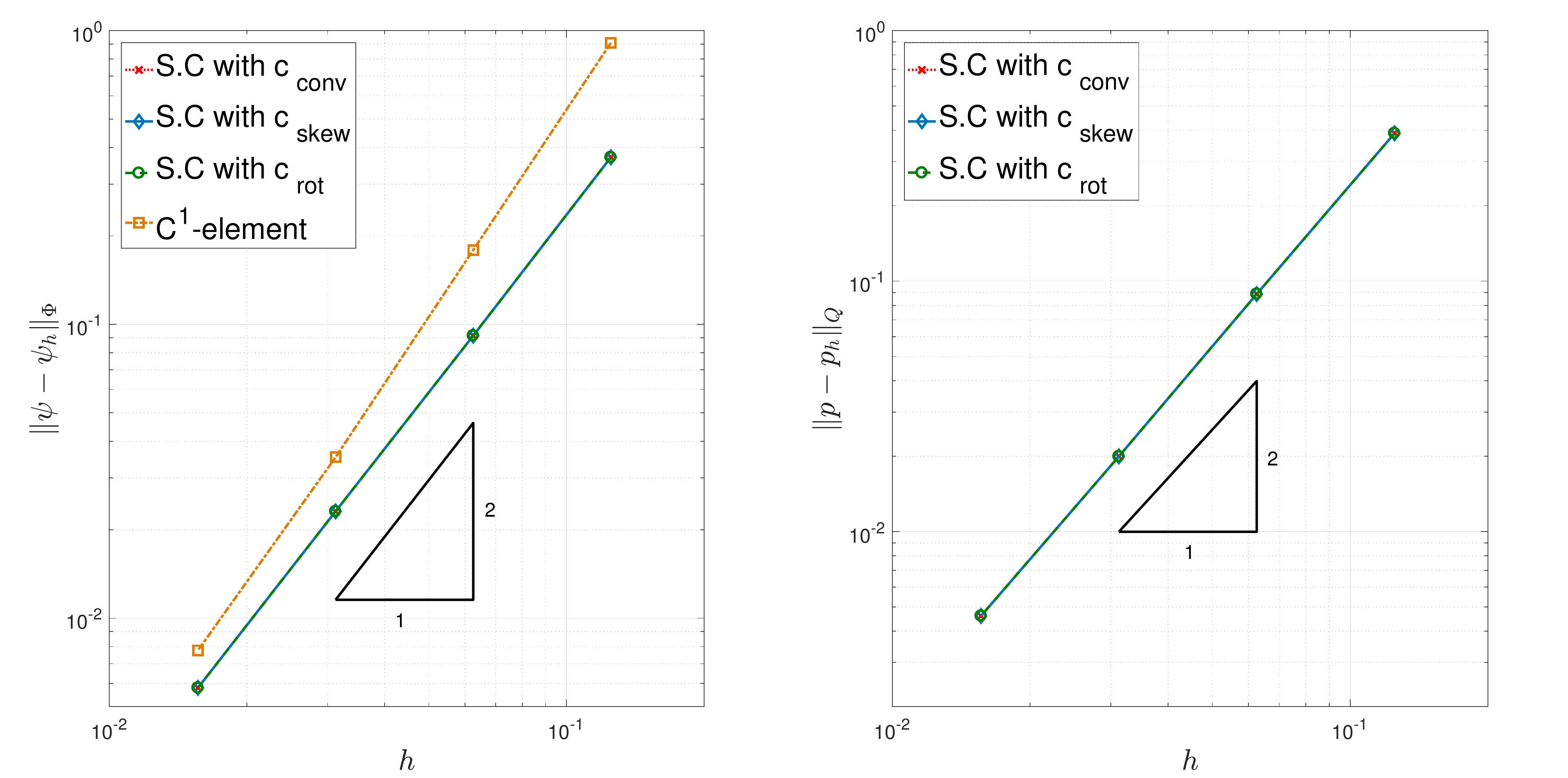} 
\caption{Test \ref{test1}.  Errors computed with the VEM \eqref{eq:vem-curl} and
\eqref{eq:ns stream vem}, meshes $\mathcal{V}_h$.}
\label{fig:h-voronoi}
\end{figure}
\begin{figure}[!h]
\centering
\includegraphics[scale=0.3]{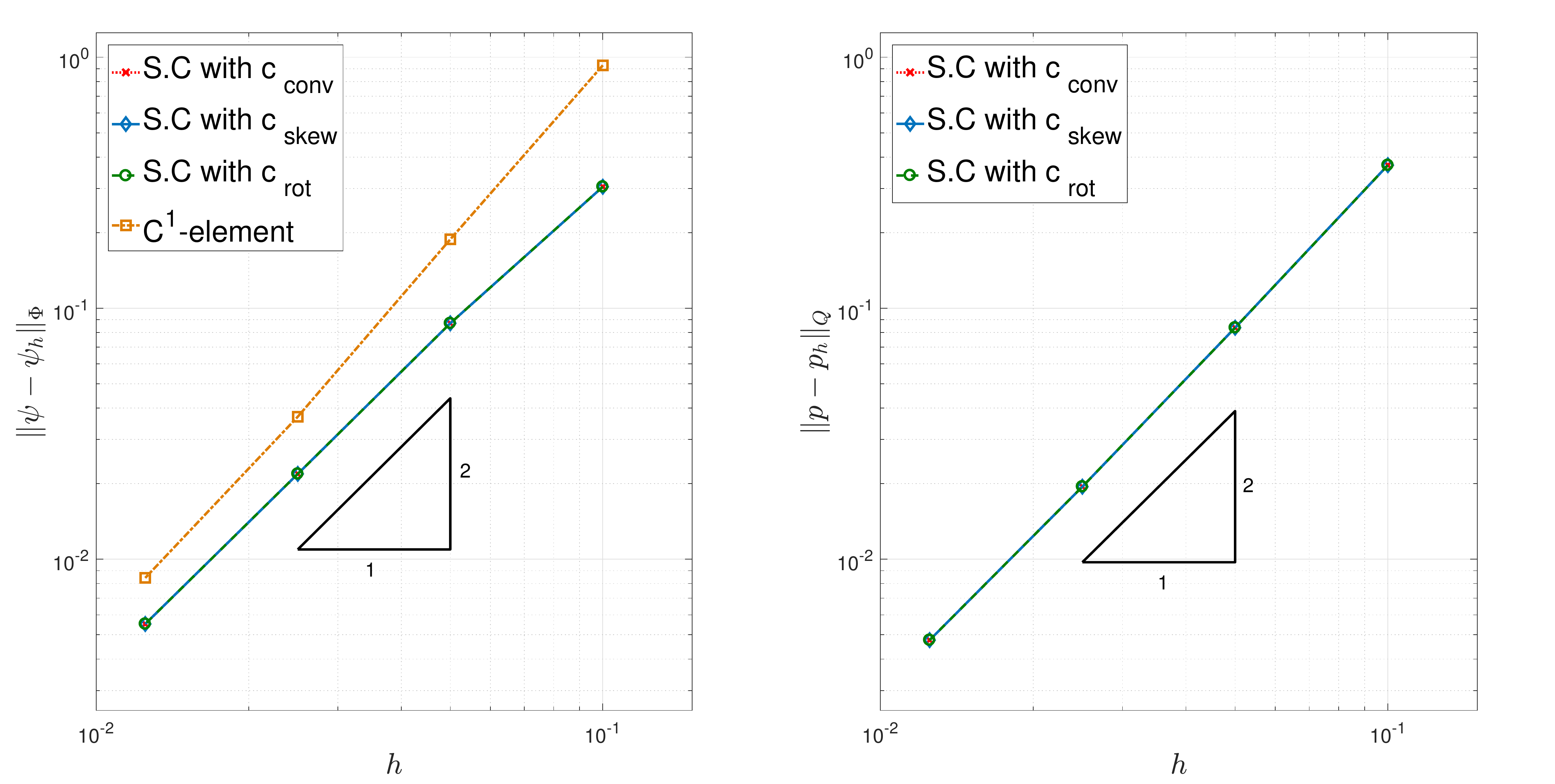} 
\caption{Test \ref{test1}.  Errors computed with the VEM \eqref{eq:vem-curl} and 
\eqref{eq:ns stream vem}, meshes $\mathcal{Q}_h$.}
\label{fig:h-quad}
\end{figure}
We notice that the theoretical predictions of Section \ref{sec:5} are confirmed. Moreover, we observe that the virtual element methods obtained with the three different trilinear forms exhibit almost identical results, at least for this example and with the adopted meshes. In Figures \ref{fig:h-voronoi} (left) and \ref{fig:h-quad} (left) we also depict the error for the direct $C^1$ discretization of the stream formulation \eqref{eq:ns stream vem}, that follows a similar behaviour to \eqref{eq:vem-curl}. Note that we do not compute a pressure error for scheme \eqref{eq:ns stream vem} since the computation of a discrete pressure is a more complex issue in this case, see Remark \ref{rem:C1linke}.
Finally we test the corresponding virtual element method \eqref{eq:ns virtual} with the same sequences of polygonal meshes $\mathcal{V}_h$, $\mathcal{Q}_h$. 
Table \ref{tab1} shows the results obtained respectively  with VEM \eqref{eq:ns virtual} and \eqref{eq:vem-curl} obtained considering the trilinear form $\croth(\cdot; \cdot, \cdot)$.
The results are analogous also for the other two proposed trilinear forms (not shown).
\begin{table}[!h]
\centering
\begin{tabular}{ll*{4}{c}}
\toprule
& & \multicolumn{2}{c}{velocity-pressure formulation}&\multicolumn{2}{c}{$\CC$ formulation}\\
\midrule
&                         $h$                      
& $\text{\texttt{error}}(\mathbf{u}, H^1)$   
& $\text{\texttt{error}}(p, L^2)$       
& $\text{\texttt{error}}(\psi, H^2)$        
& $\text{\texttt{error}}(p, L^2)$ 
\\
\midrule
\multirow{4}*{$\mathcal{V}_h$}                             
&\texttt{1/8}        
&\texttt{3.704032467e-1}                      
&\texttt{3.891840615e-1}                             
&\texttt{3.704032467e-1}  
&\texttt{3.891840615e-1} 
\\
&\texttt{1/16}       
&\texttt{9.153568669e-2}                      
&\texttt{8.875084726e-2}                             
&\texttt{9.153568669e-2}  
&\texttt{8.875084726e-2} 
\\
&\texttt{1/32}       
&\texttt{2.308710367e-2}
&\texttt{1.994452869e-2}                             
&\texttt{2.308710367e-2}  
&\texttt{1.994452869e-2} 
\\
&\texttt{1/64}       
&\texttt{5.791512013e-3}                       
&\texttt{4.602515029e-3}                           
&\texttt{5.791512013e-3}  
&\texttt{4.602515029e-3} 
\\
\midrule
\multirow{4}*{$\mathcal{Q}_h$}                             
&\texttt{1/10}       
&\texttt{3.047752518e-1}                      
&\texttt{3.714633884e-1}                             
&\texttt{3.047752518e-1}  
&\texttt{3.714633884e-1} 
\\
&\texttt{1/20}       
&\texttt{8.709526360e-2}                      
&\texttt{8.363888240e-2}                             
&\texttt{8.709526360e-2}  
&\texttt{8.363888240e-2} 
\\
&\texttt{1/40}       
&\texttt{2.188243443e-2}
&\texttt{1.945853612e-2}                             
&\texttt{2.188243443e-2}  
&\texttt{1.945853612e-2} 
\\
&\texttt{1/80}       
&\texttt{5.523374104e-3}                       
&\texttt{4.762632907e-3}                           
&\texttt{5.523374104e-3}  
&\texttt{4.762632907e-3} 
\\
\bottomrule
\end{tabular}
\caption{Test \ref{test1}.  Errors computed with virtual elements schemes \eqref{eq:ns virtual} and \eqref{eq:vem-curl} for  the meshes $\mathcal{V}_h$
and $\mathcal{Q}_h$.}
\label{tab1}
\end{table}
In Table \ref{tab2} we compare the number of DoFs and the condition number of the resulting linear systems (stemming from the fist iteration of the Newton method) 
for both formulations \eqref{eq:ns virtual}  and \eqref{eq:vem-curl}.
As observed in Section \ref{sec:5}, the scheme \eqref{eq:vem-curl} has the advantage of having $(2\, n_p - 2)$ less of unknowns, even when considering the reduced version (see Remark \ref{rem:reduced}) for formulation \eqref{eq:ns virtual}.
The drawback is that the condition number of the system resulting from the velocity-pressure scheme \eqref{eq:ns virtual} behaves as $h^{-2}$, while the asymptotic rate of the condition number of the linear system resulting from the scheme \eqref{eq:vem-curl} formulation is $h^{-4}$.
\begin{table}[!h]
\centering
\begin{tabular}{ll*{4}{c}}
\toprule
& & \multicolumn{2}{c}{velocity-pressure formulation}&\multicolumn{2}{c}{$\CC$ formulation}\\
\midrule
&                         $h$                      
& $\text{\texttt{n\_DoFs}}$   
& $\text{\texttt{condition number}}$       
& $\text{\texttt{n\_DoFs}}$        
& $\text{\texttt{condition number}}$ 
\\
\midrule
\multirow{4}*{$\mathcal{V}_h$}                             
&\texttt{1/8}        
&\texttt{585}                      
&\texttt{1.274770181e+3}                             
&\texttt{459}  
&\texttt{1.063189235e+5} 
\\
&\texttt{1/16}       
&\texttt{2573}                      
&\texttt{5.052943797e+3}                             
&\texttt{2063}  
&\texttt{7.870747143e+5} 
\\
&\texttt{1/32}       
&\texttt{10757}
&\texttt{2.347797950e+4}                             
&\texttt{8711}  
&\texttt{1.840718952e+7} 
\\
&\texttt{1/64}       
&\texttt{43833}                       
&\texttt{9.919686982e+4}                           
&\texttt{35643}  
&\texttt{3.030371659e+8} 
\\
\midrule
\multirow{4}*{$\mathcal{Q}_h$}                             
&\texttt{1/10}       
&\texttt{621}                      
&\texttt{2.395280050e+3}                             
&\texttt{423}  
&\texttt{9.117317719e+3} 
\\
&\texttt{1/20}       
&\texttt{2641}                      
&\texttt{1.151514288e+4}                             
&\texttt{1843}  
&\texttt{8.438089452e+4} 
\\
&\texttt{1/40}       
&\texttt{10881}
&\texttt{5.217365003e+4}                             
&\texttt{7683}  
&\texttt{1.299512460e+6} 
\\
&\texttt{1/80}       
&\texttt{44161}                       
&\texttt{2.121983675e+5}                           
&\texttt{31363}  
&\texttt{2.079841110e+7} 
\\
\bottomrule
\end{tabular}
\caption{Test \ref{test1}.  Number of DoFs and
condition numbers for the virtual elements schemes
\eqref{eq:ns virtual} and \eqref{eq:vem-curl} for
the meshes $\mathcal{V}_h$
and $\mathcal{Q}_h$.}
\label{tab2}
\end{table}
\end{test}

\begin{test}\label{test2}
In this test, inspired by \cite{benchmark},
we consider a problem  for the Navier--Stokes
with the property of having the velocity and
stream solution in the corresponding  discrete spaces.  
In particular we consider the Navier--Stokes
\eqref{eq:ns stream} on the disk $\Omega_{\rm D}$  
with  viscosity $\nu = 1$ and with
the load 
$\mathbf{f}$ chosen such that the analytical velocity-pressure solution 
and the corresponding stream function solution are respectively 
\begin{gather*}
\mathbf{u}(x,y) = \begin{pmatrix}
x^2 + y^2 \\ -2 \, x y
\end{pmatrix} \qquad 
p(x,y) = x^3 y^3 - \frac{1}{16} \,,
\\
\psi(x, y) = x^2 y + \frac{1}{3} y^3 \,.
\end{gather*}  
We notice that the stream function solution $\psi$ belongs to the discrete spaces $\Phi_h$ and $\widetilde{\Phi}_h$ (cf. \eqref{eq:Phi_h} and \eqref{eq:Phi_ht} respectively) 
and the corresponding velocity $\uu = \CC \, \psi$ belongs to $[\Pk_k(\Omega)]^2 \subset \VV_h$.
In the light of Remark \ref{rm:trilinear}, the methods \eqref{eq:vem-curl} obtained considering the trilinear forms $\cconvh(\cdot; \cdot, \cdot)$ or $\croth(\cdot; \cdot, \cdot)$ provide a higher order approximation error, that is $h^{k+2}$ (instead of $h^k$ as a standard inf-sup stable FEM would do). Note that instead an exactly divergence-free FEM would acquire machine precision velocity error on this test, but this comes at the price of the additional complications of this methods.
Figures \ref{fig:linke-tri} and \ref{fig:linke-vor} plot the results obtained with  the sequence of triangular meshes $\mathcal{T}_h$ and  mapped Voronoi meshes $\mathcal{W}_h$ considering the scheme \eqref{eq:vem-curl} (with all possible choices of the trilinear forms)  and the scheme \eqref{eq:ns stream vem}.
The error for the $C^1$ discretization of the stream formulation \eqref{eq:ns stream vem} provides the same higher order of convergence $h^{k+2}$. The error for scheme \eqref{eq:ns stream vem} in the case of the finest Voronoi mesh is not depicted since the Newton iterations did not converge in that case, probably an indication on the reduced robustness of such scheme when compared to \eqref{eq:vem-curl}.
\begin{figure}[!h]
\centering
\includegraphics[scale=0.3]{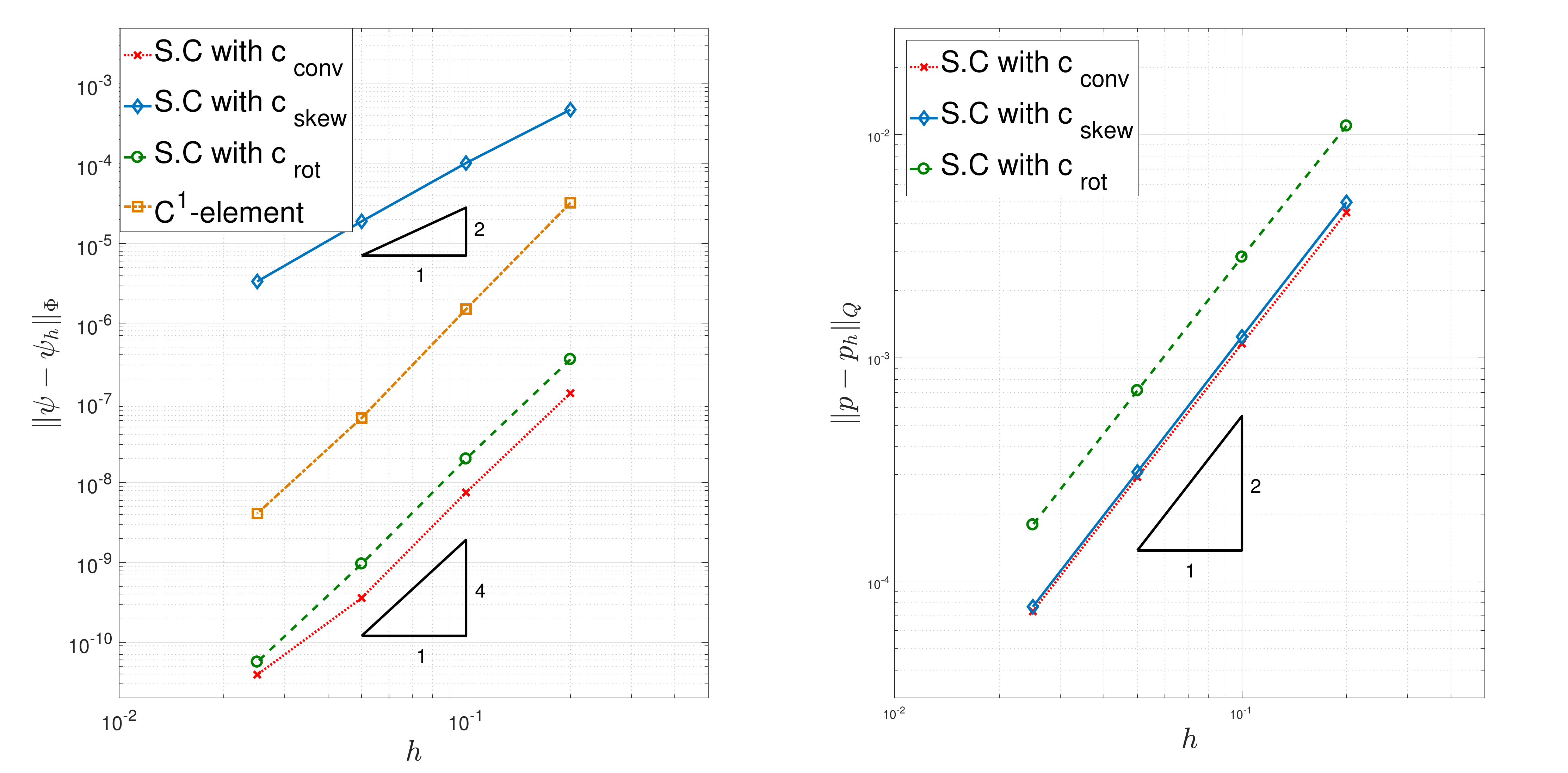} 
\caption{Test \ref{test2}.  Errors computed with the VEM \eqref{eq:vem-curl} and
\eqref{eq:ns stream vem}, meshes $\mathcal{T}_h$.}
\label{fig:linke-tri}
\end{figure}

\begin{figure}[!h]
\centering
\includegraphics[scale=0.3]{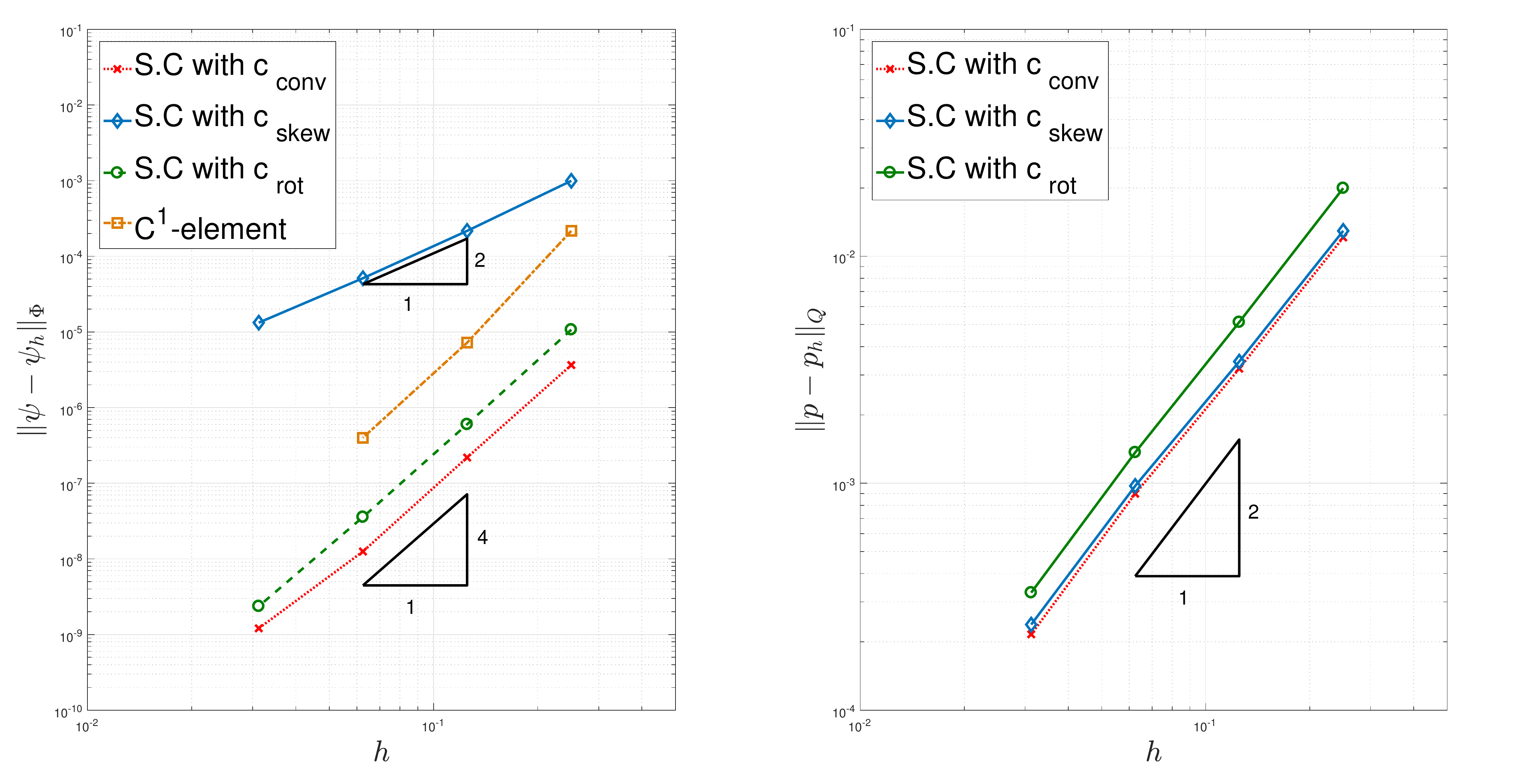} 
\caption{Test \ref{test2}.  Errors computed with the VEM \eqref{eq:vem-curl}
and
\eqref{eq:ns stream vem},  meshes $\mathcal{W}_h$.}
\label{fig:linke-vor}
\end{figure}

\end{test}

\section{Appendix: virtual elements stream formulation}
\label{sec:6}

In the present section we  briefly present a stream-function formulation approach, that is based on a direct approximation with $C^1$ Virtual Elements \cite{Brezzi-Marini:2012, antonietti-bdv-scacchi-verani:2016} of the continuous stream function formulation of the problem \eqref{eq:ns stream variazionale}. This approach is not associated to a discrete Stokes complex and therefore lacks an equivalent discrete velocity-pressure formulation.

Recalling \eqref{eq:Psi_h}, we consider on each element
$E \in \Omega_h$ the finite dimensional local virtual space
\begin{equation}
\label{eq:Phi_ht}
\widetilde{\Phi}_h^E := \biggl\{  
\phi \in \Psi_h^E  \quad \text{s.t.} \quad
\left( \phi - \Pi^{\nabla^2,E}_k  \phi, \,  \widehat{p}_{k-1} \right)_{E} = 0 \quad 
\text{for all $\widehat{p}_{k-1} \in  \widehat{\Pk}_{k-1 \setminus k-3}(E)$}
\biggr\} \,.
\end{equation}

We here mention only the main properties of the virtual space
$\widetilde{\Phi}_h^E$ and refer to \cite{Brezzi-Marini:2012, antonietti-bdv-scacchi-verani:2016}
for a deeper description:
\begin{itemize}
\item \textbf{dimension:}
it holds $\Pk_{k+1}(E) \subset \widetilde{\Phi}_h^E$ and 
$
\dim\left( \widetilde{\Phi}_h^E \right) = 2n_E k + \frac{(k-1)(k-2)}{2}  
$,
that is the same dimension of $\Phi_h^E$ (cf. \eqref{eq:dimension_Phi});
\item \textbf{degrees of freedom:}
the linear operators $\mathbf{D_{\widetilde{\Phi}}}$:
$\mathbf{D_{\Phi}1}$, $\mathbf{D_{\Phi}2}$, $\mathbf{D_{\Phi}3}$,
$\mathbf{D_{\Phi}4}$, $\widetilde{\mathbf{D_{\Phi}5}}$
(see Remark~\ref{rm:brezzi marini}) constitute a set of
DoFs for $\widetilde{\Phi}_h^E$;
\item \textbf{projections:}
the DoFs $\mathbf{D_{\widetilde{\Phi}}}$ allow us to compute exactly 
(c.f. \eqref{eq:Pn2_k^E} and \eqref{eq:P0_k^E})
\begin{gather*}
{\Pi}_{k+1}^{\nabla^2, E} \colon \widetilde{\Phi}_h^E \to \Pk_{k+1}(E) \,, \qquad
{\Pi}_{k-1}^{0, E} \colon \Delta (\widetilde{\Phi}_h^E) \to \Pk_{k-1}(E) \,, \qquad
{\Pi}_{k-1}^{0, E} \colon \widetilde{\Phi}_h^E \to \Pk_{k-1}(E) \,,
\\
{\Pi}_{k-1}^{0, E} \colon \gr (\widetilde{\Phi}_h^E) \to [\Pk_{k-1}(E)]^2 \,, \qquad
{\Pi}_{k-1}^{0, E} \colon \cc (\widetilde{\Phi}_h^E) \to [\Pk_{k-1}(E)]^2 \,,
\end{gather*}
in the sense that, given any $\phi_h \in \widetilde{\Phi}_h^E$,
we are able to compute the polynomials
${\Pi}_{k+1}^{\nabla^2, E} \phi_h$, 
${\Pi}_{k-1}^{0, E} \Delta \phi_h$,
${\Pi}_{k-1}^{0, E} \phi_h$,
${\Pi}_{k-1}^{0, E}  \gr \phi_h$,
and 
${\Pi}_{k-1}^{0, E}  \cc \phi_h$
only using, as unique information, the degree of
freedom values $\mathbf{D_{\widetilde{\Phi}}}$ of $\phi_h$.
\end{itemize}
The global virtual element space is obtained as usual by combining the local spaces $\Phi_h^E$
accordingly to the local degrees of freedom 
\begin{equation} %
\label{eq:Phi_global_ht}
\widetilde{\Phi}_h := \{ \phi \in H^2_0(\Omega) \quad \text{s.t} \quad \phi_{|E} \in \widetilde{\Phi}_h^E  \quad \text{for all $E \in \Omega_h$} \} \,.
\end{equation} 
%

Now we briefly describe the
construction of a discrete computable version of the bilinear
form $\widetilde{a}(\cdot, \cdot)$  given in \eqref{eq:forma at}
and trilinear form $\widetilde{c}(\cdot; \cdot, \cdot)$ (cf. \eqref{eq:forma crott}).
First, we decompose into local contributions
$\widetilde{a}(\cdot, \cdot)$ and $\widetilde{c}(\cdot; \cdot, \cdot)$
by considering:
\begin{equation*}
\widetilde{a} (\psi, \, \phi) =: 
\sum_{E \in \Omega_h} \widetilde{a}^E (\psi, \, \phi)\,, \qquad
\widetilde{c}(\zeta; \, \psi, \phi) =: \sum_{E \in \Omega_h} 
\widetilde{c}^E(\zeta; \, \psi, \phi) \qquad 
\text{for all $\zeta, \psi, \phi \in \Phi$.}
\end{equation*}
Following a standard procedure in the VEM framework,
we define a computable discrete local bilinear form
%
$\widetilde{a}_h^E(\cdot, \cdot) 
\colon \widetilde{\Phi}_h^E \times \widetilde{\Phi}_h^E \to \R$
approximating the continuous form $\widetilde{a}_h^E(\cdot, \cdot) $, and defined by
\begin{equation}
\label{eq:a_ht^E def}
\widetilde{a}_h^E(\psi, \, \phi) := 
\left(\Pi_{k-1}^{0, E} \Delta \psi, \, \Pi_{k-1}^{0, E} \Delta  \phi \right)_E + 
\widetilde{\mathcal{S}}^E \left( \left(I - \PNdue \right) \psi, \, \left(I -\PNdue \right) \phi \right)
\end{equation}
for all $\psi, \phi \in \widetilde{\Phi}_h^E$, where the (symmetric)
stabilizing bilinear form
$\widetilde{\mathcal{S}}^E \colon \widetilde{\Phi}_h^E \times \widetilde{\Phi}_h^E \to \R$,
satisfies the stability condition
\begin{equation}
\label{eq:St^E}
\alpha_* \|\phi\|^2_{\Phi,E} \leq  
\widetilde{\mathcal{S}}^E(\phi, \, \phi) \leq 
\alpha_*  \|\phi\|^2_{\Phi,E} \qquad \text{for all $\phi \in \widetilde{\Phi}_h^E$ such that $\PNdue \phi = 0$}
\end{equation}
with $\alpha_*$ and $\alpha^*$  positive  constants independent of the element $E$.
We notice that the bilinear form  $\widetilde{a}_h^E(\cdot, \, \cdot)$
is not coercive, since its kernel consists of all harmonic polynomial functions.
For what concerns the approximation of the local trilinear form $\widetilde{c}^E(\cdot; \, \cdot, \cdot)$, we simply set
\begin{equation}
\label{eq:c_ht^E}
\widetilde{c}_h^E(\zeta_h; \, \psi_h, \, \phi_h) := 
\int_E \left[ \left(\Pi^{0, E}_{k-1} \, \Delta \zeta_h  \right)  \left(\Pi_{k}^{0, E} \, \CC \psi_h \right) \right] \cdot \left(\Pi^{0, E}_{k} \gr \psi_h \right) \, {\rm d}E 
\end{equation}
for all $\zeta_h, \psi_h, \phi_h \in \Phi_h$, which is computable by the DoFs $\mathbf{D_{\widetilde{\Phi}}}$.

We define the global approximated forms $\widetilde{a}_h(\cdot, \cdot)$ 
and $\widetilde{c}_h(\cdot; \, \cdot, \cdot)$
 by simply summing the local contributions:
\begin{equation}
\label{eq:global_ht}
\widetilde{a}_h(\psi_h, \, \phi_h) := 
\sum_{E \in \Omega_h}  \widetilde{a}_h^E(\psi_h, \, \phi_h) \, \qquad 
\widetilde{c}_h(\zeta_h; \, \psi_h, \, \phi_h) := 
\sum_{E \in \Omega_h}  \widetilde{c}_h^E(\zeta_h; \, \psi_h, \, \phi_h)\,, 
\end{equation}
for all $\zeta_h, \psi_h, \phi_h \in \Phi_h$. 
Finally the right hand side is computed as explored in \eqref{eq:right new}.

Referring to~\eqref{eq:Phi_global_ht}, ~\eqref{eq:global_ht} and \eqref{eq:right new}, the virtual element stream formulation of the Navier--Stokes equation is given by
\begin{equation}
\label{eq:ns stream vem}
\left\{
\begin{aligned}
& \mbox{ find $\psi_h \in \widetilde{\Phi}_h$ such that}& &\\
&   \nu \, \widetilde{a_h}(\psi_h, \, \phi_h) 
\, +  \, 
\widetilde{c_h}(\psi_h; \, \psi_h, \,  \phi_h)  = ((\cc \, \ff)_h, \, \phi_h) 
\qquad & &\text{for all  $\phi_h \in \widetilde{\Phi}_h$.}
\end{aligned}
\right.
\end{equation}

\begin{remark}\label{rem:C1linke}
Note that the $C^1$ discretization of Navier--Stokes~\eqref{eq:ns stream variazionale}
does not allow to reconstruct an approximation of the continuous pressure
$p$ analogous to that given in \eqref{eq:pressure}.
Indeed the core idea behind \eqref{eq:pressure} is to exploit the
Stokes complex associated to $\Phi_h$ and $\VV_h$ (cf. \eqref{eq:Phi_h_global}
and \eqref{eq:V_h}, respectively).
\end{remark}

\section*{Acknowledgements}
The authors L. Beir\~ao da Veiga and G. Vacca were
partially supported by the European Research Council through
the H2020 Consolidator Grant (grant no. 681162) CAVE,
Challenges and Advancements in Virtual Elements.
This support is gratefully acknowledged.
The author D. Mora was partially supported by
CONICYT-Chile through project FONDECYT 1180913 and
project AFB170001 of the PIA Program: Concurso Apoyo
a Centros Cient\'ificos y Tecnol\'ogicos de Excelencia con
Financiamiento Basal.

\addcontentsline{toc}{section}{\refname}
\bibliographystyle{plain}
\bibliography{biblio}

\end{document}